\documentclass{article}
\usepackage{amssymb, amsmath, latexsym, amsthm}
\usepackage[activeacute, activegrave, english]{babel}
\usepackage{hyperref}
\usepackage{color}
\hypersetup{urlcolor=blue} 
\providecommand{\keywords}[1]{\textbf{Keywords:} #1}
\usepackage[sort&compress,square,comma,authoryear]{natbib}
\usepackage[toc]{appendix}

\usepackage{longtable}

\usepackage{authblk}

\usepackage{tikz}
\usetikzlibrary{positioning,arrows,calc}
\tikzset{
modal/.style={>=stealth',shorten >=1pt,shorten <=1pt,auto,node distance=1.5cm,
semithick},
real world/.style={double,circle,draw,thick,align=center},
world/.style={circle,draw,minimum size=0.5cm,fill=gray!15},
world1/.style={circle,draw,minimum size=0.5cm,fill=white!15}
point/.style={circle,draw,inner sep=0.5mm,fill=black},
reflexive above/.style={->,loop,looseness=7,in=120,out=60},
reflexive below/.style={->,loop,looseness=7,in=240,out=300},
reflexive left/.style={->,loop,looseness=7,in=150,out=210},
reflexive right/.style={->,loop,looseness=7,in=30,out=330}
}

\usepackage{subcaption}

\newtheorem{defi}{Definition}
\newtheorem{prop}{Proposition}
\newtheorem{theor}{Theorem}

\newtheorem{lemma}{Lemma}

\newtheorem{example}{Example}

\date{}

\author{
  Ekaterina Kubyshkina\footnote{Department of Philosophy, University of Milan, \texttt{ekaterina.kubyshkina@unimi.it}}
    \and
  Marcio Kl\'{e}os Pereira\footnote{Department of Philosophy, Federal University of Maranh\~{a}o, \texttt{marcio.kleos@ufma.br}}
  \and
  Mattia Petrolo\footnote{LASIGE Computer Science and Engineering Research Centre, University of Lisbon, \texttt{mpetrolo@fc.ul.pt}}
}

\title{Ignorance as an excuse, formally}

\begin{document}

\maketitle

\begin{abstract}

There is a lively debate in the current literature on epistemology on which type of ignorance may provide a moral excuse. A good candidate is the one in which an agent has never thought about or considered as true a proposition $p$. From a logical perspective, it is usual to model situations involving ignorance by means of epistemic logic. However, no formal analysis has been provided for ignorance as an excuse. We fill this gap by proposing an original logical setting for modelling this type of ignorance. In particular, we introduce a complete and sound logic in which excusable ignorance is expressed as a primitive modality. This logic is characterized by Kripke semantics with possibly incomplete worlds. Moreover, to consider the conditions of a possible change of an agent's ignorance, we will extend the setting to public announcement logic equipped with a novel update procedure.

\end{abstract}


\section{Introduction}
\label{intro}

Since Aristotle's \textit{Nichomachean Ethics}, the question of whether ignorance provides a moral excuse for the actions of an agent has attracted moral philosophers and has given rise to many lively discussions.\footnote{Some of these views are considered in \citet{Smith1983}.} Recent debates in ethics and social epistemology focus on the exact type of ignorance that might provide an excuse. The question can be approached from different perspectives, depending on the specific objective: which situations count as genuine cases of ignorance, which kinds of ignorance are blameworthy, whether ignorance can provide an agent with an excuse, and whether we are considering blame and responsibility from a moral or a legal standpoint. In what follows, we will adopt a specific position concerning moral (rather than legal) responsibility and ignorance as an excuse, drawing on the account developed by \citet{Peels2014, Peels2017} . In particular, Peels considers four types of ignorance that are susceptible of providing a moral excuse:

\begin{itemize}

\item \textit{Disbelieving Ignorance:} $S$ is disbelievingly ignorant that $p$ iff (i) it is true that $p$, and (ii) $S$ disbelieves that $p$.\footnote{Following \citet[p. 484]{Peels2014}, disbelieving that $p$ is understood as believing that the negation of $p$ is true while $p$ is true.}

\item \textit{Suspending Ignorance:} $S$ is suspendingly ignorant that $p$ iff (i) it is true that $p$, and (ii) $S$ suspends belief in $p$.

\item \textit{Deep Ignorance:} $S$ is deeply ignorant that $p$ iff (i) it is true that $p$, and (ii) $S$ neither believes that $p$, nor disbelieves that $p$, nor suspends belief in $p$.

\item \textit{Warrantless Ignorance:} $S$ is warrantlessly ignorant that $p$ iff (i) it is true that $p$, (ii) $S$ believes that $p$, and (iii) $S$ does \textit{not} know that $p$.

\end{itemize}

In the formal setting introduced later, we focus on a common feature of situations involving disbelieving and deep ignorance: the agent's ignorance of a proposition $p$ is based on the fact that they do not consider $p$ as possibly true in any possible scenario. In the case of disbelieving ignorance, the proposition is considered as false for all that the agent can think about. In the case of deep ignorance, the agent does not consider the proposition at all: it is neither true, nor false to them.

\citet{Peels2014} argues that only disbelieving and deep ignorance can count as types of ignorance that fully excuse the agent's actions, where a \textit{full excuse} is defined as follows:

``Some person $S$'s ignorance that $p$ fully excuses her for the actualization of some state of affairs $\Sigma$ iff (i) $S$ fails to meet an all-things-considered obligation to prevent the actualization of $\Sigma$ or to (not) do something which would have prevented the actualization of $\Sigma$, and (ii) due to $S$'s ignorance that $p$, $S$ is blameless for the actualization of $\Sigma$.'' (\citet[p. 482-484]{Peels2014}) 

However, not all excuses qualify as full excuses. Some types of ignorance, in particular, suspending ignorance, leave an agent at least partly blameworthy, and thus allow only for partial excuses rather than full ones. In our framework, we set these cases aside and restrict our analysis to types of ignorance that provide a full moral excuse.

\citet{Peels2014} identifies four kinds of propositions that are relevant when it comes to ignorance as an excuse. They are:

\begin{enumerate}

\item \textit{Ignorance of One's Obligation.} $S$ is ignorant that she has an all-things-considered obligation $O$ (not) to actualize $\Sigma$ or (not) to do something which would have prevented the actualization of $\Sigma$.  \citep[p. 485]{Peels2014}

\item \textit{Ignorance of One's Ability to Meet One's Obligation.} $S$ is ignorant that she is able to meet her all-things-considered obligation $O$ (not) to actualise $\Sigma$ or (not) to do something which would have prevented the actualisation of $\Sigma$. \citep[p. 487]{Peels2014} 

\item \textit{Ignorance of How to Meet One's Obligation}. $S$ is ignorant that $X_{1}$, $X_{2}$, ..., or $X_{n}$ is a sufficiently good means that is available to her to meet her all-things considered obligation $O$ (not) to actualize $\Sigma$ or (not) to do something which would have prevented the actualization of $\Sigma$. \citep[p. 488]{Peels2014}

\item \textit{Lack of foresight.} One can know or truly believe that one has an obligation to $\phi$, that one is able to $\phi$, and even how to $\phi$ and yet be ignorant that $\neg \phi$-ing will result in the actualization of $\Sigma$. \citep[p. 489]{Peels2014}

\end{enumerate}
 
Following Peels' account, it is the disbelieving ignorance and deep ignorance of propositions 1 -- 4 that provide a full moral excuse.



To illustrate a situation of fully excusable ignorance, let us borrow an example from \citet[p. 482]{Peels2014}:

\begin{quotation}

If I give my daughter a piece of chocolate that, unbeknownst to me, was poisoned by a maniac who happened to choose my house for his malicious action, and I have no indication whatsoever to think that the chocolate is poisoned, then, it seems, I am not blameworthy at all for giving her that piece of chocolate.

\end{quotation}

In this example, the agent who gives to his daughter a poisoned piece of chocolate can be fully excused for this action because they do not consider the proposition `The chocolate bar might be poisoned' as true in any possible scenario. Compare this situation with a modified version of the example, given by \citet[p. 483]{Peels2014}:

\begin{quotation}

Imagine that I heard on the news that some maniac is poisoning people's chocolate bars in my neighbourhood. I notice that my chocolate bar is opened. But then, I know, we often leave opened chocolate bars in the desk and finish them later. I nonetheless decide to give it to my daughter. Imagine that it is poisoned. It seems that in that case I am blameworthy for giving it to my daughter; I act recklessly in doing so and violate an objective obligation not to give it to her. Still, it seems, I am not as blameworthy as I would have been if I had known that it was poisoned. Thus, my ignorance that it is poisoned makes me less blameworthy that I would have been if I had not been ignorant, but I am still blameworthy to some degree.

\end{quotation}

In this second situation, the agent considers the proposition `The chocolate bar might be poisoned' as possibly true, because he heard the news. By giving the chocolate bar to his daughter, the agent chooses to dismiss this possibility. For this reason, his ignorance cannot serve as a full excuse for his action. Indeed, this is a case of suspending ignorance which, following Peels, does not qualify as fully excusable ignorance. \footnote{Since this example corresponds to suspending ignorance, it follows \textit{a fortiori} that it is not appropriate to apply the excusing conditions outlined in propositions 1 -- 4. Recall that, for Peels, only cases of disbelieving ignorance or deep ignorance can ground a full moral excuse.}

Distinguishing situations of fully excusable ignorance from others that do not provide a full excuse is crucial when evaluating an agent's responsibility for their actions. Even though a formal setting that captures this distinction might prove to be very helpful in this task, one cannot find any formal representation of fully excusable ignorance in the literature.  
Our aim is to address this gap by developing a logical framework for representing excusable ignorance and its dynamic features. More specifically, we formalize cases of disbelieving and deep ignorance through the operator $I$ introduced in \citep{Kubyshkina2019}. It should be emphasized, however, that the framework presented here presupposes that the instances of disbelieving and deep ignorance captured by $I$ satisfy at least one of Peels' excusing conditions, as specified in propositions 1 -- 4.

The rest of this article is structured as follows. In Section \ref{FIL system section} we consider a logic for factive ignorance introduced by \citet{Kubyshkina2019} and we modify its semantics by allowing possible worlds to be incomplete. We show that the resulting system constitutes an adequate formal representation of fully excusable ignorance and we prove its soundness and completeness. In Section \ref{PAL section} we extend this setting with public announcements, which permit us to consider not only static situations involving fully excusable ignorance but also the possible change of the state of ignorance of an agent. In particular, we first discuss why the standard public announcement update is not suitable for our purposes. Second, we introduce an original alternative definition of the update procedure and prove its completeness. Finally, we reconsider the examples of \citet{Peels2014} in this dynamic setting and show how an announcement can turn excusable into non-excusable ignorance.

\section{A logic of excusable ignorance}
\label{FIL system section}

Several recent works in epistemic logic have focused on finding a way to model the notion of ignorance. Two interesting proposals can be found in \citep{Hoek2004} and \citep{Steinsvold2008}. The former defines ignorance as `not-knowing whether' and represent it via a modal operator that can be defined as $\neg K \phi \wedge \neg K \neg \phi$ in standard epistemic logic. \citet{Steinsvold2008} defines ignorance as `not-knowing the truth' and considers a primitive modal operator definable as $\phi \wedge \neg K \phi$. 
However, the applicability of both operators seems to be too large for expressing strictly speaking excusable ignorance. This is due to the fact that both definitions incorporate the case of suspending ignorance which is a case of non-excusable ignorance according to the characterization of Peels.\footnote{This is not a criticism of the Kripke semantics \textit{per se}, but of the definitions of the operators.}
Moreover, neither of the operators properly represent deep ignorance: once an agent is ignorant of a proposition, there exists at least one world in which the negation of this proposition is considered to be true. In what follows, we provide an alternative setting that is suitable to represent fully excusable ignorance.

\subsection{Syntax and semantics}

Our starting point is the logic \textbf{ELI} that was introduced by \citet{Kubyshkina2019}. This logic, which is characterized by standard Kripke semantics, contains a primitive modality $I$ that is defined as follows:

\begin{itemize}

\item $\mathcal{M}, w \models I \phi \Leftrightarrow$  for all $w' $ that are not $ w $ and such that $Rww'$, $\mathcal{M}, w' \not \models \phi $ and $\mathcal{M}, w \models \phi$.

\end{itemize}

Once the underlying logic is classical propositional logic, this definition is equivalent to the following:

\begin{itemize}

\item $\mathcal{M}, w \models I \phi \Leftrightarrow$  for all $w' $ that are not $ w $ and such that $Rww'$, $\mathcal{M}, w'  \models \neg \phi $ and $\mathcal{M}, w \models \phi$.

\end{itemize}

In accordance with this definition, an agent is ignorant of $\phi$ in a world $w$ if and only if $\phi$ is true but $\neg \phi$ holds in all accessible worlds from $w$.\footnote{Clearly, the worlds in which $\neg \phi$ is valid are not the world $w$ itself, because of the fact that the underlying semantics is two-valued and thus $\phi$ and $\neg \phi$ cannot be valid in the same world.} This corresponds to the situation of disbelieving ignorance, as described in the introduction. However, the use of $I$ operator does not permit one to model situations involving deep ignorance. This setting does not distinguish a situation in which a proposition is considered to be false (disbelieving ignorance), and another in which it is considered to be neither true nor false (deep ignorance). To capture this difference, we propose to use a three-valued setting, in which each proposition can take one of three values: `$1$' (true), `$0$' (false), and `$\emptyset$' (neither true, nor false). Semantically, it is common to use Kleene's logic for these purposes.\footnote{Recently, \citet{Bonzio2023} have introduced a definition of ``severe ignorance'' in a three-valued setting. The authors use Bochvar's logic and interpret the third value as ``meaningless,'' which is different from our purposes.}

In Kleene's strong logic (see \citet{K1938, Kleene1952}), the valuation function for formulas with propositional operators is defined as in Table \ref{Tables}. 

\begin{longtable}{ccc}
\caption{Kleene's operators}\\
\label{Tables}
\begin{tabular}{c|c}
$\phi$&$\neg \phi$\\\hline
$1$&$0$\\
$\emptyset$&$\emptyset$\\
$0$&$1$
\end{tabular}
&
\begin{tabular}{c|ccc}
$\phi \wedge \psi$&$1$&$\emptyset$&$0$\\\hline
$1$&$1$&$\emptyset$&$0$\\
$\emptyset$&$\emptyset$&$\emptyset$&$0$\\
$0$&$0$&$0$&$0$
\end{tabular}
&
\begin{tabular}{c|ccc}
$\phi \vee \psi$&$1$&$\emptyset$&$0$\\\hline
$1$&$1$&$1$&$1$\\
$\emptyset$&$1$&$\emptyset$&$\emptyset$\\
$0$&$1$&$\emptyset$&$0$
\end{tabular}
\end{longtable}

One can define Kleene's implication in accordance with standard definition: $\phi \rightarrow \psi := \neg \phi \vee \psi$. Being definable via disjunction and negation, Kleene's implication is already included in the setting we are describing. However, by doing so one would lose the deduction theorem.\footnote{Note that this carries both practical and theoretical disadvantages. Some authors have pointed to the deduction theorem as a key desideratum that the definition of formal deduction should meet (see, e.g., \citet{Montague1956}).} To avoid this, we add an implication defined as in Table 2. This implication is as close as possible to classical logic: implicative formulas can only take classical values, $1$ or $0$. From an intuitive perspective, this choice of implication contributes to say that $\phi \rightarrow \psi$ is true if and only if whenever the antecedent is true, the consequent is also true. By adding this implication we get an example of a well studied algebraic setting, namely, a weak Heyting algebra which satisfies several natural properties (see, e.g., \citet{Celani2005}, \citet{Bezhanashvili2011}, \citet{Celani2012}). Other types of implications can be considered, but for the purposes of this article, we leave this investigation aside.


\begin{longtable}{c}
\caption{Implication}\\
\label{implication}
\begin{tabular}{c|ccc}
$\phi \rightarrow \psi$&$1$&$\emptyset$&$0$\\\hline
$1$&$1$&$0$&$0$\\
$\emptyset$&$1$&$1$&$1$\\
$0$&$1$&$1$&$1$
\end{tabular}
\end{longtable}


Let us define the syntax and the semantics for our logic. Given a non-empty set $\texttt{Prop}$ of propositional variables and $p \in \texttt{Prop}$, the  language $\mathcal{L}$ is defined by the following grammar:

$$\phi :: = p \,\, | \,\, \neg \phi \,\, | \,\, \phi \wedge \phi \,\, | \,\, \phi \rightarrow \phi \,\, | \,\, I \phi$$

Other propositional operators are defined in a standard way: $\phi \vee \psi \Leftrightarrow \neg (\neg \phi \wedge \neg \psi)$ and $\phi \leftrightarrow \psi \Leftrightarrow (\phi \rightarrow \psi) \wedge (\psi \rightarrow \phi)$. A formula $I \phi$ has to be read as `the agent is excusably ignorant that $\phi$'. For simplicity, in what follows we introduce a single-agent setting but it can be extended to a multi-agent framework in a standard way.

In contrast from \citet{Kubyshkina2019}, we interpret the language $\mathcal{L}$ on Kripke semantics with \textit{incomplete} worlds; that is, worlds which can contain either $p$, or $\neg p$, or neither $p$ nor $\neg p$. 
The first semantic definition of our logic \textbf{LEI} (Logic of Excusable Ignorance) follows the lines of \citet[p. 284-285]{Odintsov2012}.

\begin{defi}
\label{frames1}

A \textbf{LEI}-frame $\mathcal{F} = (W, R)$ is a tuple where $W$ is a set of possible worlds and $R \subseteq W \times W$ is an accessibility relation. A \textbf{LEI}$^{1}$-model $\mathcal{M}^{1} = (\mathcal{F}, V)$, is a tuple where $\mathcal{F}$ is a Kripke frame and $V$ is a valuation function.  $V(p, w)$ assigns to each propositional variable $p$ either $\{1\}$ (true), or $\{0\}$ (false), or $\emptyset$ (neither true, nor false) at the world $w$. The valuation $V$ extends to all formulas of the language $\mathcal{L}$ as follows:

\begin{itemize}

\item $V(\phi \wedge \psi, w) = \{1\}$ iff $V(\phi, w) = \{1\}$ and $V(\psi, w) = \{1\}$;

\item[] $V(\phi \wedge \psi, w) = \{0\}$ iff $V(\phi, w) = \{0\}$ or $V(\psi, w) = \{0\}$;

\item[] $V(\phi \wedge \psi, w) = \emptyset$ otherwise.

\item $V(\neg \phi, w) =\{1\}$ iff $V(\phi, w) = \{0\}$;

\item[] $V(\neg \phi, w) =\{0\}$ iff $V(\phi, w) = \{1\}$;

\item[] $V(\neg \phi, w) = \emptyset$ otherwise.

\item $V(\phi \rightarrow \psi, w) = \{1\}$ iff ($V(\phi, w) = \{1\}$ and $V(\psi, w) = \{1\}$) or $V(\phi, w) \not = \{1\}$;

\item[] $V(\phi \rightarrow \psi, w) = \{0\}$ otherwise.

\item $ V(I \phi, w) = \{1\}$ iff for all $w'$ that are not $w$ and such that $Rww'$, $V(\phi, w') \not = \{1\}$ and $V( \phi, w) = \{1\}$;

\item[] $V(I \phi, w) = \{0\}$ otherwise.

\end{itemize}

We say that $\phi$ is valid on $\mathcal{M}^{1}$ and write $\mathcal{M}^{1} \models \phi$ if $V(\phi, w) = \{1\}$ for all $w$ of $\mathcal{M}^{1}$.  If for all $\mathcal{M}^{1}$ based on $\mathcal{F}$ we have $\mathcal{M}^{1} \models \phi$, then we say that $\phi$ is valid on $\mathcal{F}$ and write $\mathcal{F} \models \phi$.


\end{defi}

In accordance with this  definition, fully modalized formulas (i.e., formulas in which all propositional variables are under the scope of $I$) can only take classical values $\{1\}$ or $\{0\}$.  Intuitively, this corresponds to the fact that an agent is either excusably ignorant of a proposition, or she is not. If the latter is the case, then she is either inexcusably ignorant of a proposition, or she is not ignorant of it at all.

For the sake of the completeness proof, it is convenient to have another semantic definition of \textbf{LEI} that is closer to the standard Kripke semantics.


\begin{defi}[Frames, Models, and Satisfaction]
\label{frames}

A \textbf{LEI}-frame $\mathcal{F} = (W, R)$ is a tuple where $W$ is a set of possible worlds, and $R \subseteq W \times W$ is an accessibility relation. A \textbf{LEI}-model $\mathcal{M} = (\mathcal{F}, v)$, is a tuple where $\mathcal{F}$ is a \textbf{LEI}-frame and  $v$ is a valuation function such that, for each atomic proposition $p$, $v(p) \longrightarrow P(W)$, $v(\neg p) \longrightarrow P(W)$ and $v(p) \cap v(\neg p) = \emptyset$.
Given a model $\mathcal{M}$ and a formula $\phi$, we say that $\phi$ is true in $\mathcal{M}$ at world $w$, written $\mathcal{M}, w \models \phi$ if:

\begin{itemize}
\item $\mathcal{M}, w \models p$ if $w \in v(p)$;
\item $\mathcal{M}, w \models \neg p$ if $w \in v(\neg p)$;
\item $\mathcal{M}, w \models \neg \neg \phi$ if $\mathcal{M}, w \models \phi$;
\item $\mathcal{M}, w \models \phi \wedge \psi$ if $\mathcal{M}, w \models \phi$ and $\mathcal{M}, w \models \psi$;
\item $\mathcal{M}, w \models \neg (\phi \wedge \psi)$ if $\mathcal{M}, w \models \neg \phi$ or $\mathcal{M}, w \models \neg \psi$;
\item $\mathcal{M}, w \models \phi \rightarrow \psi$ if $\mathcal{M}, s \models \phi$ implies $\mathcal{M}, w \models \psi$;
\item $\mathcal{M}, w \models \neg (\phi \rightarrow \psi)$ if $\mathcal{M}, w \not \models \phi \rightarrow \psi$;
\item $\mathcal{M}, w \models I\phi$ if for all $w'$ that are not $w$ and such that $Rww'$, $\mathcal{M}, w'  \not \models  \phi$ and $\mathcal{M}, w \models \phi$;
\item $\mathcal{M}, w \models \neg I\phi$ if either there exists $w'$ that is not $w$ and such that $Rww'$, $\mathcal{M}, w' \models  \phi$, or $\mathcal{M}, w \not \models \phi$.
\end{itemize} 

We say that $\phi$ is valid on $\mathcal{M}$ and write $\mathcal{M} \models \phi$ if $\mathcal{M}, w \models \phi$ for all $w$ in $W$. If for all $\mathcal{M}$ based on $\mathcal{F}$ we have $\mathcal{M} \models \phi$, then we say that $\phi$ is valid on $\mathcal{F}$ and write $\mathcal{F} \models \phi$.

\end{defi}

Even though, Definition \ref{frames} does not provide a clause for $\mathcal{M}, w \models \neg \phi$, the formulas of a form $\neg \phi$ in $\mathcal{M}$ are defined inductively on the length of the formula $\phi$.

Definition \ref{frames} is equivalent to Definition \ref{frames1}. Similarly to \citep{Odintsov2012}, we assign a model $\mathcal{M}^{1} = (W, R,V)$  to a model $\mathcal{M} = (W, R, v)$, where 

$$V(p, w) = \{1\} \text{ iff } w \in v(p);$$

$$V(p, w) = \{0\} \text{ iff } w \in v(\neg p);$$

$$V(p, w) = \emptyset \text{ iff } w \not \in v(p) \text{ and } w \not \in v(\neg p).$$

It is easy to check that this relation extends to arbitrary formulas:

$$V(\phi, w) = \{1\} \text{ iff } \mathcal{M}, w \models \phi;$$

$$V(\phi, w) = \{0\} \text{ iff } \mathcal{M}, w \models \neg \phi;$$

$$V(\phi, w) = \emptyset \text{ iff } \mathcal{M}, w \not \models \phi \text{ and } \mathcal{M}, w \not \models \neg \phi.$$

Notice that Definitions \ref{frames1} and \ref{frames} do not impose any restrictions on the accessibility relation. This indicates that our accessibility relation and the worlds are not the indistinguishability relation and the epistemic states, respectively, as in standard epistemic logic. In our reading, the fact that a world $w'$ is accessible from the world $w$ means that if the world $w$ is the actual one in which the agent reasons, then the world $w'$ is a world that contains some possible truths from the agent's perspective. The possible worlds that are accessible from $w$ represent thus  the sets of propositions that an agent considers as possibly true. For instance, an agent may hesitate to assert whether or not the name of the author of the novel \textit{Midnight's Children} is Salman Rushdie. In this case the agent considers at least two accessible worlds: the one in which the author's name is Salman Rushdie, and the one in which it is not. Moreover, if the agent has never heard of this novel, and thus had never thought about the name of its author, then they will consider neither the worlds in which the proposition ``The name of the author of the novel \textit{Midnight's Children} is Salman Rushdie'' is true nor the worlds in which it is false. However, the agent might hesitate about whether the proposition ``Kashmir is in India'' is true, that is, there may be accessible worlds in which it is true and others in which it is false. Now consider the valuation of ``The author of the novel \textit{Midnight's Children} is Salman Rushdie'' in those accessible worlds where ``Kashmir is in India'' is either true or false. With the third truth value $\emptyset$, the proposition about Salman Rushdie can be assigned this value (i.e., neither ``true'' nor ``false'') in such worlds. This captures the idea that, while the agent reflects on the truth conditions of ``Kashmir is in India,'' she does not regard the proposition about the author of \textit{Midnight's Children} as either true or false. In summary, accessible possible worlds indicate which propositions an agent considers as possibly true. If a proposition takes the value $\emptyset$ in a world, this means that its truth conditions are not considered (or are irrelevant) for the agent, unlike those propositions for which she does consider the classical truth conditions.

Note that reflexivity does not affect our understanding of the accessibility relation and accessible  worlds. As discussed in \citep{Gilbert2020}, the logic \textbf{ELI} is reflexive-insensitive (i.e., \textbf{ELI}  is a logic which is insensitive to the presence of reflexivity in the accessibility relation, see \citet{Gilbert2016}). The same observation holds for \textbf{LEI}. From this perspective, the worlds that are  possible for an agent at a point $w$ are always the worlds that are not $w$ itself. The agent always reasons on the basis of their \textit{hypotheses} of what the actual world looks like but not on the basis of what it really is.

In accordance with Definitions \ref{frames1} and \ref{frames}, the underlying logic for each world would be Kleene's strong logic. The formal system for the language containing non-modal formulas can be found in \citep{Robles2019}, where it is dubbed \textbf{Lt2}.

\begin{defi}[System \textbf{Lt2}] 
\label{KlFDE}

\

\begin{itemize}

\item The axiom schemes:

\begin{itemize}

\item[A1.] $(\phi \wedge \psi) \rightarrow \phi$

\item[A2.] $(\phi \wedge \psi) \rightarrow \psi$

\item[A3.] $((\phi \rightarrow \psi) \wedge (\phi \rightarrow \chi)) \rightarrow (\phi \rightarrow (\psi \wedge \chi))$

\item[A4.] $\phi \rightarrow (\phi \vee \psi)$

\item[A5.] $\psi \rightarrow (\phi \vee \psi)$

\item[A6.] $((\phi \rightarrow \chi) \wedge (\psi \rightarrow \chi)) \rightarrow ((\phi \vee \psi) \rightarrow \chi)$

\item[A7.] $(\phi \wedge (\psi \vee \chi)) \rightarrow ((\phi \wedge \psi) \vee (\phi \wedge \chi))$

\item[A8.] $\neg(\phi \vee \psi) \leftrightarrow (\neg \phi \wedge \neg \psi)$

\item[A9.] $\neg (\phi \wedge \psi) \leftrightarrow (\neg \phi \vee \neg \psi)$

\item[A10.] $\phi \leftrightarrow \neg \neg \phi$

\item[A11.] $((\phi \rightarrow \psi) \wedge \phi) \rightarrow \psi$

\item[A12.] $\psi \rightarrow (\phi \rightarrow \psi)$

\item[A13.] $\phi \vee (\phi \rightarrow \psi)$

\item[A14.] $\phi \rightarrow (\psi \vee \neg (\phi \rightarrow \psi))$

\end{itemize}

\item The rules of inference:

\begin{itemize}

\item[(Adj)] from $ \phi$ and $ \psi$ infer $\phi \wedge \psi$

\item[(MP)] from $\phi \rightarrow \psi$ and $ \phi$ infer $\psi$

\item[(dMP)] from $ \chi \vee (\phi \rightarrow \psi)$ and $ \chi \vee \phi$ infer $ \chi \vee \psi$

\item[(dTrans)] from $ \rho \vee (\phi \rightarrow \psi)$ and $ \rho \vee (\psi \rightarrow \chi)$ infer $ \rho \vee (\phi \rightarrow \chi)$

\item[(dECQ)] from $ \chi \vee (\phi \wedge \neg \phi)$ infer $ \chi \vee \psi$

\end{itemize}

\end{itemize}

\end{defi}

As noticed in \citep{Robles2019}, the following rules and tautologies are derivable from the axiom schemes and rules presented in Definition \ref{KlFDE}.

\begin{prop}

The following rules are derivable in \textbf{Lt2}:

\begin{itemize}

\item[(Trans)] from $ \phi \rightarrow \psi$ and $ \psi \rightarrow \chi$ infer $ \phi \rightarrow \chi$

\item[(ECQ)] from $ \phi \wedge \neg \phi$ infer $ \psi$

\item[(t1)] $\phi \rightarrow \phi$

\end{itemize}

\end{prop}

Our system of excusable ignorance \textbf{LEI} is an extension of \textbf{Lt2} and is defined as follows.

\begin{defi}[System \textbf{LEI}]
\label{FILdefi}

\

\begin{itemize}

\item All the axiom schemes and rules of inference of Definition \ref{KlFDE}

\item The axiom schemes:

\begin{itemize}

\item[(\textit{fact})] $I \phi \rightarrow \phi$

\item[($I \wedge$)] $(I\phi \wedge I\psi) \rightarrow I(\phi \vee \psi)$

\item[($emI$)] $I \phi \vee \neg I \phi$

\end{itemize}

\item The rule: 

\begin{itemize}

\item[(IR)] from $ \vdash \phi \rightarrow \psi$ infer $ \vdash \phi \rightarrow (I \psi \rightarrow I \phi)$

\end{itemize}

\end{itemize}

A derivation of \textbf{LEI} is a finite sequence of $\mathcal{L}$-formulas such that each formula is either the instantiation of an axiom scheme or the result of applying an inference rule to previous formulas in the sequence. A formula $\phi \in \mathcal{L}$ is called a theorem, noted $\vdash \phi$, if it occurs in a derivation of \textbf{LEI}. An expression $\Gamma \vdash \phi$ means that $\phi$ can be obtained from a set of formulas $\Gamma$ by applying axiom schemes and rules of \textbf{LEI}.

\end{defi}

The deduction theorem ($DT$) holds for the system \textbf{LEI}.

\begin{prop}
\label{dedt}

For any $\phi, \psi \in \mathcal{L}$:

\begin{itemize}

\item[($DT$)] $\Gamma, \phi \vdash \psi$  iff $ \Gamma \vdash \phi \rightarrow \psi$

\end{itemize}

\end{prop}

The proof of this proposition is in Appendix \ref{appDT}.

The proofs of the following propositions are in Appendix \ref{Fanprops}.


\begin{prop}
\label{Arrow need-1}

\

\begin{itemize}

\item[(T1)] $\neg \phi \rightarrow (\phi \rightarrow \psi)$

\end{itemize}

\end{prop}


\begin{prop}
\label{Arrow need}

\

\begin{itemize}

\item[(T2)] $(\neg \phi \vee \psi) \rightarrow (\phi \rightarrow \psi)$

\end{itemize}

\end{prop}


\begin{prop}
\label{prop5}

\

\begin{itemize}

\item[(R1)] from $\phi \rightarrow (\psi \rightarrow \chi)$ infer $(\phi \wedge \psi) \rightarrow \chi$

\end{itemize}

\end{prop}

\begin{lemma}[Soundness]
\label{sound}

The system \textbf{LEI} is sound with respect to the class of all frames.

\end{lemma}

The proof of the lemma is available in Appendix \ref{AppA}.

\subsection{Completeness}

We prove the completeness of \textbf{LEI} by constructing a canonical model. Let us start by introducing the following observations that are used in the completeness proof.

Uniform Substitution is a derivable rule in \textbf{LEI} and we use it in the proof of the Truth Lemma.

\begin{prop}
\label{subst}

For any formula $\alpha$ whose propositional variables are included in $\{p_1,\dots, p_n\}$, and $\beta_1,\dots, \beta_{n}$ are any well-formed formulas (wff) of \textbf{LEI}, then $\alpha[\beta_1/p_1,\dots, \beta_n/p_n]$ is the formula that results from uniformly substituting $\beta_i$ for $p_i$ in $\alpha$. 

$$\text{From } \vdash \alpha \text{ infer } \vdash \alpha[\beta_1/p_1,\dots, \beta_n/p_n] \eqno(US)$$
\end{prop}

The proof of this proposition is by a straightforward induction on the length of the formulas and can therefore be omitted.

To provide the proof of the Truth Lemma, we will need the following proposition, for which the proof is similar to that of Proposition 2 in \citep{Kubyshkina2019}.

\begin{prop}
\label{wedgegen}

For all $n \geq 1$:

$$\vdash (I \phi_{1} \wedge ... \wedge I \phi_{n}) \rightarrow I(\phi_{1} \vee ... \vee \phi_{n})\eqno(I\wedge^{gen})$$

\end{prop}

The notion of a \textit{\textbf{LEI}-theory} $\mathcal{T}$ can be defined in a standard way as a non-trivial set\footnote{A set of formulas $\mathcal{T}$ is called \textit{trivial} if every $\phi \in \mathcal{T}$. Otherwise, it is called non-trivial.} of \textbf{LEI} formulas, closed under the principles of \textbf{LEI}, and satisfying the following property: if $\phi \in \mathcal{T}$ or $\psi \in \mathcal{T}$, then $\phi \vee \psi \in \mathcal{T}$. A theory is \textit{prime} if it satisfies the property: if $\phi \vee \psi \in \mathcal{T}$, then $\phi \in \mathcal{T}$ or $\psi \in \mathcal{T}$. A theory is \textit{consistent} if for no formula $\phi$, both $\phi \in \mathcal{T}$ and $\neg \phi \in \mathcal{T}$. A theory is \textit{maximal consistent} if it is consistent and contains as many formulas as it can without becoming inconsistent. Similarly to \citep{Dunn2000} we first introduce the Extension Lemma for proving the completeness result.

\begin{lemma}[Extension Lemma]
\label{Ext}

Let $\phi \not \vdash \psi$, then there exists a consistent prime \textbf{LEI}-theory $\mathcal{T}$ such that $\phi \in \mathcal{T}$ and $\psi \not \in \mathcal{T}$.

\end{lemma}

The proof is in Appendix \ref{appLemmaExt}.

The canonical model for \textbf{LEI} is defined as follows.

\begin{defi}[Canonical model]
\label{can mod}

The canonical model $\mathcal{M}^{C}$ for \textbf{LEI} is the triple $(W^{C}, R^{C}, v^{C})$, where:

\begin{itemize}
\item $W^{C} = \{w \mid w$ is a consistent prime \textbf{LEI}-theory$\}$;
\item $R^{C}ww'$ iff for all $\phi$: if $I \phi \in w$ then $\phi \not \in w'$;
\item $v^{C}(p) = \{w \in W^{C} \mid p \in w\}$ and $v^{C}(\neg p) =  \{w \in W^{C} \mid \neg p \in w\}$.
\end{itemize}

\end{defi}

The definition of $R^{C}ww'$ is taken from \citep{Gilbert2020}.

Now we can prove the Truth Lemma.

\begin{lemma}[Truth Lemma]
\label{Truth}

For all formulas $\phi$, and all consistent prime \textbf{LEI}-theories $w$,

$$\mathcal{M}^{C}, w \models \phi \mbox{ iff } \phi \in w;$$
$$\mathcal{M}^{C}, w \models \neg \phi \mbox{ iff } \neg \phi \in w.$$

\end{lemma}

See Appendix \ref{appTruth1} for the proof. Notice that we need to distinguish the cases of $\phi$ and $\neg \phi$, because in a three-valued setting $\mathcal{M}^{C}, w \models \neg \phi \Leftrightarrow \mathcal{M}^{C}, w \not \models \phi$ does not hold.

\begin{theor}[Completeness]

The system \textbf{LEI} is sound and complete with respect to the class of all frames.

\end{theor}

\begin{proof}

Soundness is proven in Lemma \ref{sound}. Completeness follows in a standard way from Lemmas \ref{Ext} and \ref{Truth}. Let $\phi \not \vdash \psi$. By Lemma \ref{Ext}, we construct a consistent prime theory $\alpha$ with $\phi \in \alpha$ and $\psi \not \in \alpha$. By Lemma \ref{Truth}, we have $\mathcal{M}, \alpha \models \phi$ and $\mathcal{M}, \alpha \not \models \psi$.

\end{proof}


Now we are able to distinguish various possible situations of an agent's ignorance.
Consider, for instance, a model $\mathcal{M}$ as in Figure \ref{model M*}, i.e., $\mathcal{M} = (W, R, v)$ such that $W = \{w_{0}, w_{1}, w_{2}, w_{3}\}$, $R = \{(w_{0}, w_{1}), (w_{0}, w_{2}), (w_{0}, w_{3})\}$, $v(p) = \{w_{0}\}$, $v(\neg p) = \{w_{1}, w_{2}, w_{3}\}$, $v(q) = \{w_{0}\}$, $v(\neg q) = \{w_{2}\}$, $v(r) = \{w_{0}\}$, and $v(\neg r) = \emptyset$. In this model, an agent is ignorant in the world $w_{0}$ of  propositions $p$, $q$ and $r$: $\mathcal{M}, w_{0} \models I p \wedge I q \wedge I r$. However, the reasons for being ignorant of these propositions are different. In the case of $p$, the agent considers it to be a false proposition (i.e., for all that they consider, $\neg p$ holds). The agent is ignorant of $p$ in this case simply because they are wrong about the truth of $p$.\footnote {See  \citet{Gilbert2020} for a comparative analysis of the logic \textbf{ELI} and the logic of being wrong.} This case  corresponds to disbelieving ignorance. The situation is different in the case of ignorance of $r$. The agent is ignorant of the truth of $r$ because they consider it to be neither true nor false. This represents the case of deep ignorance. The third case, the case of ignorance of $q$, represents a `mixed situation'. The agent is ignorant of $q$ in this case because they do not consider it to be a true proposition but in some scenarios they consider it to be false, and in some other $q$ is irrelevant to them and they do not consider it at all. In all three cases, the agent is ignorant of a proposition by not considering it as true, which corresponds to the situations of excusable ignorance. Thus, the $I$ operator encodes the conditions for an agent to be excusably ignorant.

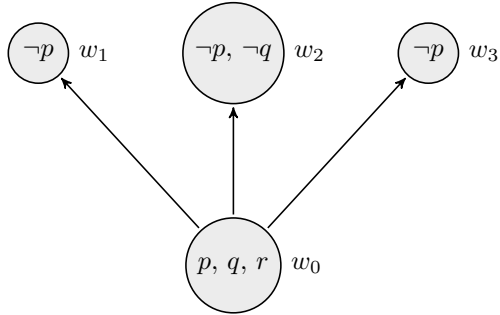
\begin{figure}
\centering
\begin{tikzpicture}[modal,world/.append style={minimum size=0.8cm}]
\node[world] (w) [label=right:$w_{0}$] {$p$, $q$, $r$};
\node[world] (s) [label=right:$w_{2}$, above=of w] {$\neg p$, $\neg q$};
\node[world] (u) [label=right:$w_{1}$, left=of s] {$\neg p$};
\node[world] (r) [label=right:$w_{3}$, right=of s] {$\neg p$};
\path[->] (w) edge (u);
\path[->] (w) edge (s);
\path[->] (w) edge (r);
\end{tikzpicture}
\caption{Various forms of ignorance}
\label{model M*}
\end{figure}

\section{Dynamic setting for \textbf{LEI}}
\label{PAL section}

The results presented in the previous section permit us to capture excusable ignorance and distinguish it from a non-excusable one. In this section, we extend the system \textbf{LEI} with public announcements to represent a possible change in the moral responsibility of an agent. First, we consider standard eliminative update procedure and explain why it cannot be applied directly to our setting. Second, we provide an alternative definition of the update procedure. On the basis of this definition, we introduce a sound and complete system. Finally, we reconsider the examples of \citet{Peels2014} in this new framework.

\subsection{Why standard \textbf{PAL} is not suitable}

The first formal settings for representing public communication were independently provided by \citet{Plaza1989} and \citet{Gerbrandy1997}. Since these pioneering works, it is usual to consider Public Announcement Logic (\textbf{PAL}) as an extension of standard modal logic obtained by adding a truthful public announcements operator $[! \,\,]$. The formulas of the form $[!\phi]\psi$ should be read as ``after every truthful announcement of $\phi$, formula $\psi$ is true.'' The semantic clause for this operator is usually defined in standard Kripke semantics, as follows (see, e.g., \cite{vanDitmarsch2008}):

\begin{defi}
\label{standard an}

Let $\mathcal{M} = (W, R, v)$ be a standard Kripke model. For any $\phi$ and $\psi$:

$\mathcal{M}, w \models [!\phi] \psi$ iff $\mathcal{M}, w \models \phi$ implies $\mathcal{M}\vert_{\phi}, w \models \psi$,

where $\mathcal{M}\vert_{\phi} = (W', R', v')$ is defined as follows:

\begin{center}

$W' = [\![ \phi ]\!]_{\mathcal{M}}$ (where $ [\![ \phi ]\!]_{\mathcal{M}} := \{w \in W \mid \mathcal{M}, w \models \phi\}$)

$R' = R \cap ( [\![ \phi ]\!]_{\mathcal{M}} \times  [\![ \phi ]\!]_{\mathcal{M}})$

$v' = v \cap  [\![ \phi ]\!]_{\mathcal{M}}.$

\end{center}

\end{defi}

With this definition, the formula $[!\phi]\psi$ is true if, and only if, whenever $\phi$ is true, $\psi$ is true after that one eliminates all the possible worlds in which $\phi$ is not true. This is why this kind of announcement is called \textit{eliminative}. 

Unfortunately, a direct application of this definition to the setting of \textbf{LEI} leads to some counterintuitive results. Let us see why. Let $\mathcal{M}_{1} = (W, R, v)$ be a \textbf{LEI}-model, such that $W = \{w_{0}, w_{1}, w_{2}\}$, $R = \{(w_{0}, w_{1}), (w_{0}, w_{2})\}$, $v(p) = \{w_{0}\}$ and $v(\neg p) = \{w_{1}\}$ (see Figure \ref{mod pal}). By definition of $I$, we have ignorance of $p$ in $w_{0}$; that is, $\mathcal{M}_{1}, w_{0} \models I p$. Assume that there is an eliminative announcement of $p$. This will lead to a model $\mathcal{M}_{1}\vert_{p} = (W', R', v')$, where $W' = \{w_{0}\}$, $R = \emptyset$, $v'(p) = \{w_{0}\}$ and $v'(\neg p) = \emptyset$ (see Figure \ref{mod pal upd}). The update by $p$ eliminates all the worlds in $\mathcal{M}_{1}$, except for the world $w_{0}$, which means that $\mathcal{M}_{1}, w_{0} \models I p \rightarrow [!p] I p$ (i.e., after a truthful announcement of $p$, the agent remains ignorant of $p$). Moreover, after the announcement of $p$, the agent starts to be ignorant of all the truths of $w_{0}$, even the ones of which they were not ignorant before; that is, $\mathcal{M}_{1}, w_{0} \models \neg I \top \rightarrow [!p] I \top$, where $\top$ stands for any tautology. Both of these situations are extremely counterintuitive and show that the eliminative announcement of Definition \ref{standard an} is not suitable to model a possible change of an agent's state after a truthful announcement is made.

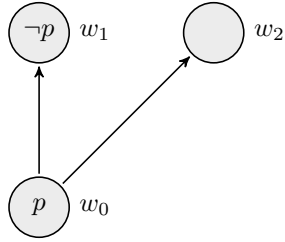
\begin{figure}
\centering
\begin{tikzpicture}[modal,world/.append style={minimum size=0.8cm}]
\node[world] (w) [label=right:$w_{0}$] {$p$};
\node[world] (u) [label=right:$w_{1}$, above=of w] {$\neg p$};
\node[world] (s) [label=right:$w_{2}$, above=of w, right=of u] {$$};
\path[->] (w) edge (u);
\path[->] (w) edge (s);
\end{tikzpicture}
\caption{Model $\mathcal{M}_{1}$}
\label{mod pal}
\end{figure}

\begin{figure}
\centering
\begin{tikzpicture}[modal,world/.append style={minimum size=0.8cm}]
\node[world] (w) [label=right:$w_{0}$] {$p$};
\end{tikzpicture}
\caption{Model $\mathcal{M}_{1}|_{p}$}
\label{mod pal upd}
\end{figure}

\begin{figure}
\centering
\begin{tikzpicture}[modal,world/.append style={minimum size=0.8cm}]
\node[world] (w) [label=right:$w_{0}$] {$p$};
\node[world] (s) [label=right:$w_{2}$, above=of w, right=of u] {$$};
\path[->] (w) edge (s);
\end{tikzpicture}
\caption{Model $\mathcal{M}_{1}|'_{p}$}
\label{mod pal st}
\end{figure}

Note also that Definition \ref{standard an} is provided for bivalent semantics, where non-truth coincides with falsity. On this basis, one may suggest to redefine  $[!\phi]\psi$ not as eliminating the worlds that do not contain $\phi$, but as eliminating the worlds that contain $\neg \phi$. In this case, the announcement of $p$ in $\mathcal{M}_{1}$ leads to a model $\mathcal{M}_{1}|'_{p}$ as in Figure \ref{mod pal st}. This strategy partially solves the problem for $\mathcal{M}_{1}$. For instance, $\mathcal{M}_{1}, w_{0} \not \models \neg I \top \rightarrow [!p] I \top$, where $\top$ stands for any tautology of \textbf{LEI}. However, we still have $\mathcal{M}_{1}, w_{0} \models I p \rightarrow [!p] I p$. Moreover, we would face the same problem in a model in which an agent is ignorant of $p$ because they consider $\neg p$ to be true in all accessible worlds. The eliminative updates seem to eliminate too much information from the initial model. Thus, our desideratum is to  define an update procedure and corresponding updated models in such a way that the propositional information of the initial model is preserved.

\subsection{Introducing the announcements}

The underlying idea of the update procedure that we introduce in this section is twofold. First, after an announcement, an agent should not start to be ignorant of truths of which they were not ignorant in the initial model.  Second, the agent should not be ignorant of the announcement once it was made, that is, the agent should take the announcement into consideration. The first condition is met by defining the updated model in such a way that it preserves all of the worlds, accessibility relations and propositional valuations of the initial model. The second condition is met by defining a new world in the updated model that includes the content of the announcement. Moreover, we do not restrict our definition to truthful announcements but to announcements that are consistent with the set of true propositions. Let us introduce this idea formally.

We define a language $\mathcal{L}^{up}$ for \textbf{LEI} with public announcements  as follows:

$$ \phi ::= p \mid  \neg \phi \mid \phi \wedge \phi \mid \phi \rightarrow \phi \mid I \phi \mid [\phi]\phi$$

This language is interpreted by Definition \ref{frames} extended with the following conditions for the update procedure.

Let us use the notation $w = \{ \chi  \mid \texttt{P} \}$ to indicate that, given a model $\mathcal{M}$, the world $w \in \mathcal{M}$ is such that exactly those formulas $\chi$ satisfying property $ \texttt{P}$ are valid in $w$; that is, $\mathcal{M}, w \models \chi$ whenever $ \texttt{P}$, and for no other formula $\chi_{1}$ does $\mathcal{M}, w \models \chi_{1}$ hold.

\begin{defi}
\label{upd-1}

Let $\mathcal{M} = (W, R, v)$ be a \textbf{LEI}-model. $Cn(\phi)$ stands for the class of all the semantic consequences of some formula $\phi \in \mathcal{L}^{up}$. Then,

\begin{itemize}

\item $\mathcal{M}, w \models [\phi] \psi$ iff consistency\footnote{Consistency here is understood in the following sense: for no formula $\psi$ in a set $A$, both $\psi \in A$ and $\neg \psi \in A$.} of $Cn(\phi) \cup \{\chi \mid \mathcal{M}, w \models \chi\}$ implies that $\mathcal{M} \vert^{w}_{\phi} , w \models \psi$, 

where $\mathcal{M} \vert^{w}_{\phi} = (W', R', v')$ is defined as follows:

$W' = W \cup \{w^{w}_{\phi}\}$ such that $w^{w}_{\phi} = \{Cn(\phi) \cup \chi \mid  \mathcal{M}, w \models \neg I \chi \wedge \chi\}$

each $R' = R \cup \{(w, w^{w}_{\phi})\} \cup\{(w^{w}_{\phi}, w') \mid w' \in W$ and $(w, w') \in R\}$

$v' = v \cup v^{*}$, such that $v^{*}(p) = \{w' \in W' | p \in w'\}$ and $v^{*}(\neg p) =  \{w' \in W' | \neg p \in w'\}$.

\item $\mathcal{M}, w \models \neg [\phi] \psi$ iff $\mathcal{M}, w \not \models [\phi] \psi$.

\end{itemize}

\end{defi}

The first interesting feature of this definition is that $\mathcal{M}, w \models \neg [\phi] \neg \psi$ does not require the truth of $\phi$ in $w$, as is the case for $\mathcal{M}, w \models \neg [!\phi] \neg \psi$ if one uses Definition \ref{standard an}. This choice is due to the fact that we do not aim to restrict announcements to only the true ones. In our setting, a non-true proposition can be announced if it does not contradict the truths of the world in which the announcement takes place.  


%
%
%
%
%
%
%
%




The second original feature of Definition \ref{upd-1}  is that the updated model is always updated with respect to some world. Consider an updated model $\mathcal{M}\vert^{w}_{\phi}$. The basic idea of its construction is that it contains all of the worlds and accessibility relations of the model $\mathcal{M}$, and, additionally  a world $w^{w}_{\phi}$ that contains all of the consequences of $\phi$ and all of the truths of which an agent is not ignorant in a given world $w$. By adding the accessibility relation from $w$ to $w^{w}_{\phi}$ we secure that, after an announcement is made, the agent is not ignorant of the content of this announcement. By adding the accessibility relations from $w^{w}_{\phi}$ to all worlds that are accessible from $w$, we secure that for all $\chi$ s.t. $\mathcal{M}^{w}_{\phi}, w^{w}_{\phi} \models \chi$ that are true and of which the agent is not ignorant in $w$ in the initial model, we have  $\mathcal{M}^{w}_{\phi}, w^{w}_{\phi} \models \neg I \chi$.

Let us give some examples of updated models. 

\begin{example}

Consider a model $\mathcal{M}_{2} = (W, R, v)$ defined as in Figure \ref{model 2}. In this model,  in $w_{0}$, an agent is ignorant of $p$ and they are not ignorant of a true proposition $q$; that is, $\mathcal{M}_{2}, w_{0} \models I p$ and $\mathcal{M}_{2}, w_{0} \models q \wedge \neg I q$. Assume that an announcement of $p$ is made with respect to the world $w_{0}$. The model updated by $p$ with respect to $w_{0}$ will contain all of the worlds and accessibilities of $\mathcal{M}_{2}$, plus a new world $w^{w_{0}}_{p}$ (see Figure \ref{model 2 upd}). The world $w^{w_{0}}_{p}$ contains the content of the announcement (i.e., it contains $p$), as well as all the true formulas of which the agent is not ignorant in $w_{0}$ of $\mathcal{M}_{2}$ (i.e., $w^{w_{0}}_{p}$ contains $q$).  It is clear that in the updated model we have $\mathcal{M}_{2}\vert^{w_{0}}_{p}, w_{0} \models \neg I p \wedge \neg I q$; that is, the agent is not ignorant of $p$ anymore, and they preserved their non-ignorance of $q$.

\end{example}

\begin{figure}
\centering
\begin{tikzpicture}[modal,world/.append style={minimum size=0.8cm}]
\node[world] (w) [label=right:$w_{0}$] {$p$, $q$};
\node[world] (s) [label=right:$w_{2}$, above=of w] {$\neg p$, $q$};
\node[world] (u) [label=right:$w_{1}$, left=of s] {\phantom{$\neg p$}};
\node[world] (r) [label=right:$w_{3}$, right=of s] {$\neg p$};
\path[->] (w) edge (u);
\path[->] (w) edge (s);
\path[->] (w) edge (r);
\end{tikzpicture}
\caption{Model $\mathcal{M}_{2}$}
\label{model 2}
\end{figure}
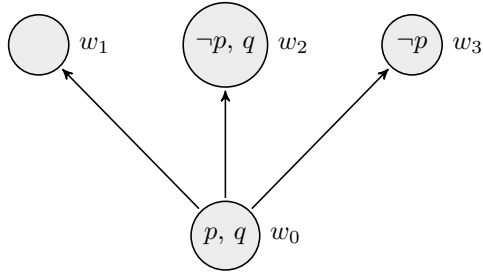

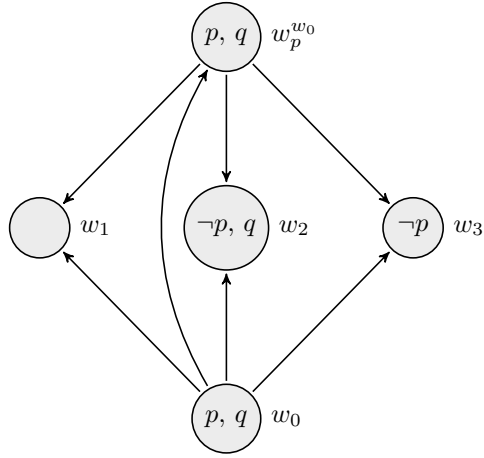
\begin{figure}
\centering
\begin{tikzpicture}[modal,world/.append style={minimum size=0.8cm}]
\node[world] (w) [label=right:$w_{0}$] {$p$, $q$};
\node[world] (s) [label=right:$w_{2}$, above=of w] {$\neg p$, $q$};
\node[world] (u) [label=right:$w_{1}$, left=of s] {\phantom{$\neg p$}};
\node[world] (r) [label=right:$w_{3}$, right=of s] {$\neg p$};
\node[world] (t) [label=right:$w^{w_{0}}_{p}$, above=of s] {$p$, $q$};
\path[->] (w) edge (u);
\path[->] (w) edge (s);
\path[->] (w) edge (r);
\path[->] (w) edge[bend left=30] (t);
\path[->] (t) edge (s);
\path[->] (t) edge (u);
\path[->] (t) edge (r);
\end{tikzpicture}
\caption{Model $\mathcal{M}_{2}|^{w_{0}}_{p}$}
\label{model 2 upd}
\end{figure}

The next two examples aim at showing how announcements of modal formulas function in our setting.

\begin{example}

Consider a model $\mathcal{M}_{2} = (W, R, v)$ defined as in Figure \ref{model 2}. Assume that an announcement of $I p$ is made with respect to the world $w_{0}$, that is ignorance of $p$ is announced to the agent. The model updated by $I p$ with respect to $w_{0}$ will in fact be the same as in Figure \ref{model 2 upd}, except we change the indexes ``$p$'' to ``$Ip$.''  One can see that $\mathcal{M}_{2}\vert^{w_{0}}_{Ip}, w^{w_{0}}_{Ip} \models Ip$, which in turn means that $\mathcal{M}_{2}\vert^{w_{0}}_{Ip}, w^{w_{0}}_{Ip} \models p$ by factivity of $I$. Moreover, $\mathcal{M}_{2}\vert^{w_{0}}_{Ip}, w_{0} \models \neg Ip$, which shows that an agent whose ignorance was announced is not ignorant anymore.

\end{example}

\begin{example}

Consider a model $\mathcal{M}_{2} = (W, R, v)$ defined as in Figure \ref{model 2}. Assume that an announcement of $\neg I q$ is made with respect to the world $w_{0}$. The updated model will be as in Figure \ref{model 2 upd 2}. Clearly, $\mathcal{M}_{2}\vert^{w_{0}}_{\neg I q}, w_{0} \models \neg Iq$, as it was the case in $w_{0}$, that is announcing that one is non-ignorant does not change the agent's state. Moreover, this example clarifies why the created world $w^{w_{0}}_{\neg I q}$ needs to access the worlds accessible from $w_{0}$. This permits to state $\mathcal{M}_{2}\vert^{w_{0}}_{\neg I q}, w^{w_{0}}_{\neg Iq} \models \neg I q$. If $w^{w_{0}}_{\neg Iq}$ were not to access $w_{2}$, the world  $w^{w_{0}}_{\neg Iq}$ could not validate $\neg I q$.

\end{example}

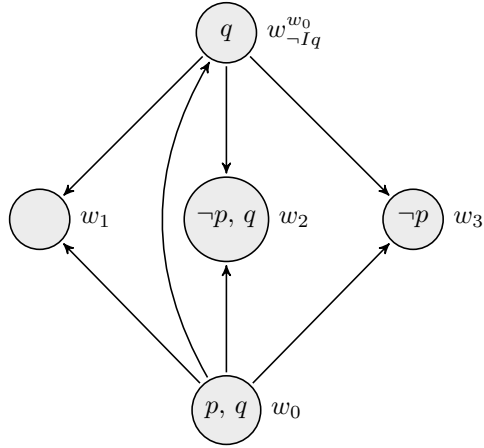
\begin{figure}
\centering
\begin{tikzpicture}[modal,world/.append style={minimum size=0.8cm}]
\node[world] (w) [label=right:$w_{0}$] {$p$, $q$};
\node[world] (s) [label=right:$w_{2}$, above=of w] {$\neg p$, $q$};
\node[world] (u) [label=right:$w_{1}$, left=of s] {\phantom{$\neg p$}};
\node[world] (r) [label=right:$w_{3}$, right=of s] {$\neg p$};
\node[world] (t) [label=right:$w^{w_{0}}_{\neg I q}$, above=of s] {$q$};
\path[->] (w) edge (u);
\path[->] (w) edge (s);
\path[->] (w) edge (r);
\path[->] (w) edge[bend left=30] (t);
\path[->] (t) edge (s);
\path[->] (t) edge (u);
\path[->] (t) edge (r);
\end{tikzpicture}
\caption{Model $\mathcal{M}_{2}|^{w_{0}}_{\neg I q}$}
\label{model 2 upd 2}
\end{figure}

%
%
%
%
%
%

Now, we can introduce the system \textbf{LEI$^{up}$}, which is characterized by the semantics described above.

\begin{defi}
\label{FILup}

\

\begin{itemize}

\item All of the axiom schemes and rules of inference of Definition \ref{FILdefi};

\item The axiom schemes:

\begin{itemize}

\item[$(AI)$] $(\phi \wedge I \neg \psi \wedge [\phi] I \psi) \rightarrow I (\psi \vee \neg \psi)$




\item[$(dA\rightarrow)$] $[\phi](\psi \rightarrow \chi) \leftrightarrow ([\phi]\psi \rightarrow [\phi]\chi)$

\item[$(nI)$] $[\phi]\neg I \phi$

\item[$(nA1)$] $\neg [\phi]\psi \rightarrow [\phi](\psi \rightarrow (p \wedge \neg p))$

\item[$(nA2)$] $\neg \phi \rightarrow [\phi] (p \wedge \neg p)$

\item[$(dA\vee)$] $[\phi](\psi \vee \chi) \leftrightarrow ([\phi]\psi \vee [\phi]\chi)$

\item[$(emA)$] $[\phi]\psi \vee \neg [\phi] \psi$

\item[$(INV)$] $(p \rightarrow [\phi]p) \wedge (\neg p \rightarrow [\phi] \neg p)$

\item[$(nAp1)$] $\neg [\phi] \neg p \rightarrow (p \vee [\phi] (p \rightarrow (q \wedge \neg q)))$

\item[$(nAp2)$] $\neg [\phi] p \rightarrow (\neg p \vee [\phi] (\neg p \rightarrow (q \wedge \neg q)))$

\item[$(uA)$] $([\phi] \psi \wedge [\phi] \neg \psi) \rightarrow (\phi \rightarrow (p \wedge \neg p))$


\end{itemize}

\item The rules: 

\begin{itemize}

\item[$(nec)$] $\mbox{From }   \vdash \phi \mbox{ infer }  \vdash  [\psi] \phi$

\item[$(intA1)$] $\mbox{From } \vdash \phi \rightarrow \psi \mbox{ infer } 
\vdash [\phi] \neg I \psi $

\item[$(intA2)$] $\mbox{From } \vdash (\phi \wedge \psi) \rightarrow \chi \mbox{ infer } 
\vdash (\psi \wedge \neg I \psi) \rightarrow [\phi] \neg I \chi $

\item[$(CN)$] $\mbox{From } \Gamma \vdash I \psi \wedge \neg [\phi] I \psi \mbox{ infer } \Gamma' \vdash \phi \rightarrow \psi$, where $\Gamma' = \{\chi \mid \chi \wedge \neg I \chi \in \Gamma \}$.



\end{itemize}

\end{itemize}

\end{defi}

Some clarifications on the meaning of the rule ($CN$) are in order. Given a context $\Gamma$ in which $\psi$ is ignored, the announcement of $\phi$ that eliminates this ignorance allows us to conclude that $\psi$ is a logical consequence of $\phi$ in the context of all the non-ignored truths of $\Gamma$ (denoted $\Gamma'$). In other words, if an announcement removes an agent's ignorance, then the proposition that the agent was initially ignorant of must be a logical consequence of the announcement combined with all truths that were not ignored in the original context.

The soundness of  \textbf{LEI$^{up}$} is proven in a standard way.

\begin{theor}[Soundness]
\label{SoundUP}

The system \textbf{LEI$^{up}$} is sound with respect to the class of all frames.

\end{theor}

The proof can be found in Appendix \ref{appSound}.

The following principle is derivable in a standard way from $(dA\rightarrow)$:

\begin{prop}
\label{prop10}

The following is a theorem of \textbf{LEI$^{up}$}:

\begin{itemize}

\item[$(dA\wedge)$] $ [\phi](\psi \wedge \chi) \leftrightarrow ([\phi]\psi \wedge [\phi]\chi)$

\end{itemize}

\end{prop}

The proof can be found in Appendix \ref{appProp10}.

The following proposition is a generalization of $(intA2)$. The proof is in the Appendix \ref{appintA2gen}.

\begin{prop}
\label{intA2gen}

\

\begin{itemize}

\item[$(intA2)^{gen}$] From $(\phi \wedge \psi_{1} \wedge  ... \wedge \psi_{n}) \rightarrow \chi$ infer 

$(\psi_{1} \wedge \neg I \psi_{1} \wedge ... \wedge \psi_{n} \wedge \neg I \psi_{n}) \rightarrow [\phi] \neg I \chi$.

\end{itemize}

\end{prop}

Note that, as in standard \textbf{PAL}, ($US$) does not hold in general. However, ($US$) holds for any propositional variable $p$ in a formula $\alpha$ if $p$ occurs neither in the content nor in the output of this announcement (i.e., if in a formula of the form $[\phi] \psi$, $p$ occurs in neither $\phi$ nor $\psi$), whenever $\alpha$ contains an announcement as a subformula. Formally:

\begin{prop}
\label{substup}

Let $\alpha$ be a formula whose propositional variables are included in $\{p_1,\dots, p_n\}$, such that if for all $\delta$ and $\gamma$ such that $[\delta] \gamma$ is a subformula of $\alpha$, $p_{1}, ..., p_{n}$ does not occur in $\delta$ and $\gamma$, and $\beta_1,\dots, \beta_{n}$ are any wff of \textbf{LEI$^{up}$}. Then, $\alpha[\beta_1/p_1,\dots, \beta_n/p_n]$ is the formula that results from uniformly substituting $\beta_i$ for $p_i$ in $\alpha$. 

$$\text{If } \vdash \alpha \text{ then } \vdash \alpha[\beta_1/p_1,\dots, \beta_n/p_n] \eqno(US^{up})$$
\end{prop}

The proof is straightforward from Proposition \ref{subst}.

In standard \textbf{PAL}, it is usual to provide reduction axiom schemes for the update procedure which (in turn) permit one to reduce the completeness proof of a given dynamic system to the completeness result of its static fragment. As one can notice from Definition \ref{FILup}, we do not provide such reduction axiom schemes for  \textbf{LEI$^{up}$} and we leave the question of their existence for further investigations.  The completeness of \textbf{LEI$^{up}$} will be provided by  the method of \textit{extended models} (see \citet{Wang2013}). First, we define a class of extended models, which is equivalent to the class of \textbf{LEI}-models. Second, we provide completeness of \textbf{LEI$^{up}$} with respect to the extended models semantics, and thus for the semantics we have defined in this subsection.

\subsection{Extended models and equivalence result}
\label{extended sec}

We show how one can reformulate the semantics in Definitions \ref{frames} and \ref{upd-1} in terms of extended models, for which the update procedure is represented in terms of accessibility relations. Let us first introduce extended models, in a similar way to \citep[pp. 119--120]{Wang2013}.

\begin{defi}[Extended model]
\label{ext model def}

An extended model $\mathcal{M}$ is a tuple $(W, R, \{R^{\phi}\mid\phi \in \mathcal{L}^{up}\}, v)$ such that:

\begin{itemize}

\item $(W, R, v)$ is as in Def. \ref{frames}.

\item For each $\phi$, $R^{\phi}$ is a (possibly empty) binary relation over $W$.

We call $\mathcal{M}^{-} = (W, R, v)$ the Kripke core of $\mathcal{M}$.

\end{itemize}

\end{defi}

\begin{defi}
\label{ext model def sat}

Let $\mathcal{M} = (W, R, \{R^{\phi}\mid\phi \in \mathcal{L}^{up}\}, v)$ be an extended model.

\begin{itemize}
\item $\mathcal{M}, w \models^{+} p$ iff $w \in v(p)$;
\item $\mathcal{M}, w \models^{+} \neg p$ iff $w \in v(\neg p)$;
\item $\mathcal{M}, w \models^{+} \neg \neg \phi$ iff $\mathcal{M}, w \models^{+} \phi$;
\item $\mathcal{M}, w \models^{+} \phi \wedge \psi$ iff $\mathcal{M}, w \models^{+} \phi$ and $\mathcal{M}, w \models^{+} \psi$;
\item $\mathcal{M}, w \models^{+} \neg (\phi \wedge \psi)$ iff $\mathcal{M}, w \models^{+} \neg \phi$ or $\mathcal{M}, w \models^{+} \neg \psi$;
\item $\mathcal{M}, w \models^{+} \phi \rightarrow \psi$ iff $\mathcal{M}, w \models^{+} \phi$ implies $\mathcal{M}, w \models^{+} \psi$;
\item $\mathcal{M}, w \models^{+} \neg (\phi \rightarrow \psi)$ iff $\mathcal{M}, w \not \models^{+} \phi \rightarrow \psi$;
\item $\mathcal{M}, w \models^{+} I\phi$ iff for all $w'$ that are not $w$ and such that $Rww'$, $\mathcal{M}, w'  \not \models^{+}  \phi$ and $\mathcal{M}, w \models^{+} \phi$;
\item $\mathcal{M}, w \models^{+} \neg I\phi$ iff either there exists $w'$ that is not $w$ and such that $Rww'$, $\mathcal{M}, w' \models^{+}  \phi$, or $\mathcal{M}, w \not \models^{+} \phi$;
\item $\mathcal{M}, w \models^{+} [\phi] \psi$ iff for all $w'$ if $R^{\phi}ww'$ then $\mathcal{M}, w' \models^{+} \psi$;
\item $\mathcal{M}, w \models^{+} \neg [\phi] \psi$ iff there exists $w'$ such that $R^{\phi}ww'$ and $\mathcal{M}, w' \not \models^{+} \psi$. 


\end{itemize} 

\end{defi}

We can interpret $\mathcal{L}^{up}$ on extended models under the semantics described in Definitions \ref{frames} and \ref{upd-1} by setting $\mathcal{M}, w \models \phi \Leftrightarrow \mathcal{M}^{-}, w \models \phi$ for any pointed extended model $(\mathcal{M}, w)$ and any formula $\phi \in \mathcal{L}^{up}$. Note that it is not necessary that $\mathcal{M}, w \models \phi \Leftrightarrow \mathcal{M}, w \models^{+} \phi$ in case of $\phi$ being a formula involving public announcement. 
For instance, it is clear that if we consider $\mathcal{M} = (W, R, v)$ s.t. $W = \{w_{0}, w_{1}\}$, $R = \{ \emptyset\}$, and $v(p) = \{w_{0}\}$, we have $\mathcal{M}, w_{0} \models [p] \neg I p$, but not $\mathcal{M}, w \models^{+} [p] \neg I p$, whenever $R^{p} = \{(w_{0}, w_{1})\}$ in the extended $\mathcal{M}$. However, we will consider a class of extended models, which satisfies the properties listed below, for which the two semantics coincide (as will be shown later):

\begin{itemize}

\item[(Func)] For any formula $\phi$: if $Cn(\phi) \cup \{\chi \mid \mathcal{M}, w 
\models^{+} \chi\}$ is consistent, then $w$ has a unique $\phi$-successor. If $Cn(\phi) \cup \{\chi \mid \mathcal{M}, w \models^{+} \chi\}$ is inconsistent, then $w$ has no outgoing $\phi$-transition.

\item[(Inv-p)] If $R^{\phi}ww'$, then for all $p \in Prop$: $w \in v(p) \Leftrightarrow w' \in v(p)$.

\item[(Inv-n)] If $R^{\phi}ww'$, then for all $p \in Prop$: $w \in v(\neg p) \Leftrightarrow w' \in v(\neg p)$.

\item[(Pr1)] If $R^{\phi}ww'$, then (1) for all $w''$ if $Rww''$ then $Rw'w''$, (2) there exists $w^{*} = Cn(\phi) \cup \{\chi \mid$  $\mathcal{M}, w \models^{+} \neg I \chi \wedge \chi\}$ such that $Rw'w^{*}$ and (3) for all $w''$ if $Rww''$ then $Rw^{*}w''$.

\item[(Pr2)] If $R^{\phi}ww'$ and $Rw'w''$, then either $w'' \subseteq Cn(\phi) \cup \{\chi \mid$ $\mathcal{M}, w \models^{+} \neg I \chi \wedge \chi\}$ or $w'' \subseteq w'''$ for some $w'''$ s.t. $Rww'''$.\footnote{Similarly to $w = \{ \chi  \mid \texttt{P} \}$, let us use $w \subseteq \{ \chi  \mid \texttt{P} \}$ to indicate that, given a model $\mathcal{M}$, the world $w \in \mathcal{M}$ is such that formulas $\chi$ satisfying property $\texttt{P}$ are valid in $w$; that is, $\mathcal{M}, w \models^{+} \chi$ whenever $\texttt{P}$.}



\end{itemize}

(Func) means that the $\phi$-update for $w$ is a partial function which depends on consistency of $Cn(\phi) \cup \{\chi \mid \mathcal{M}, w 
\models^{+} \chi\}$. (Inv-p) and (Inv-n) mean that the update should not change the valuation of the states. (Pr1) means that (1) the update for $w$ preserves all the accessibility relations from $w$; (2) that the updated $w$ accesses a new world $w^{*}$; (3) that $w^{*}$ accesses all worlds accessible from $w$. (Pr2) states that after update for $w$, the updated world sees no new information except information provided in the world $w^{*} =   Cn(\phi) \cup \{\chi \mid$ $\mathcal{M}, w \models^{+} \neg I \chi \wedge \chi\}$, that is, all the worlds accessible from $w'$ are either subsets of the world $w^{*}$, or subsets of some worlds already accessible from $w$.


The next lemma provides the relationship between the updated models and the extended models.

\begin{lemma}
\label{u-R eqv}

If for all $w''$ in $\mathcal{M}$: $\mathcal{M}^{-}, w'' \models \phi \Leftrightarrow \mathcal{M}, w'' \models^{+} \phi$, and $R^{\psi}ww'$, then

\begin{center}

$\mathcal{M}^{-}\vert^{w}_{\psi}, w \models \phi$ iff $\mathcal{M},w' \models^{+} \phi$.

\end{center}

\end{lemma}

The proof can be found in Appendix \ref{appL4}.

Now we can provide the equivalence result for models defined with $\models$ and with $\models^{+}$.

\begin{theor}
\label{equiv ext}

For any $w \in W$,

$\mathcal{M}^{-}, w \models \phi$ iff $\mathcal{M}, w \models^{+} \phi$.

\end{theor}

The proof can be found in Appendix \ref{appT12}.

%
%
%
%
%
%

As pointed out above, $\mathcal{L}^{up}$ can be interpreted on extended models under the semantics described in Definitions \ref{frames} and \ref{upd-1} by setting $\mathcal{M}, w \models \phi \Leftrightarrow \mathcal{M}^{-}, w \models \phi$ for any pointed extended model $(\mathcal{M}, w)$ and any formula $\phi \in \mathcal{L}^{up}$. Thus, we also have the equivalence between two semantic settings described in Definitions \ref{frames}, \ref{upd-1} and Definition \ref{ext model def}:

\begin{theor}
\label{M and M+ eqv}

For any $w \in W$,

$\mathcal{M}, w \models \phi \Leftrightarrow \mathcal{M}, w \models^{+} \phi$.

\end{theor}


\subsection{Completeness}

To prove the completeness for \textbf{LEI$^{up}$}, we construct a canonical model as an extension of the canonical model for the system \textbf{LEI} (see Definition \ref{can mod}), plus we define $R^{\phi C}ww'$ as follows:

\begin{defi}

\

$R^{\phi C}ww'$ iff for all $\psi$: if $[\phi]\psi \in w$ then ($\psi \in w'$ and there exists $w''$ such that $R^{C}w'w''$ and $w'' = \{\phi\} \cup \{\chi \mid \neg I \chi \wedge \chi \in w\}$).

\end{defi}

The canonical relation $R^{\phi C}$ represents a $\phi$-transition, which (as will be shown later) captures the difference between $\mathcal{M}$ and $\mathcal{M}\vert^{w}_{\phi}$ as in Def. \ref{upd-1}.

The following proposition is a generalization of ($dA\wedge$).

\begin{prop}
\label{gen!}

For all $n\leq 1$:

$$\vdash [\phi] (\chi_{1} \wedge ... \wedge \chi_{n}) \leftrightarrow ([\phi]\chi_{1} \wedge ... \wedge [\phi]\chi_{n})$$

\end{prop}

The proof is straightforward from ($dA \wedge$).

Now we can introduce the Truth Lemma.

\begin{lemma}
\label{ext truth lemma}

For all formulas $\phi \in \mathcal{L}^{up}$, and all consistent prime $LEI$-theories $w$,

\begin{center}

$\mathcal{M}^{C}, w \models^{+} \phi \Leftrightarrow \phi \in w$

and

$\mathcal{M}^{C}, w \models^{+} \neg \phi \Leftrightarrow \neg \phi \in w$

\end{center}

\end{lemma}

The proof can be found in Appendix \ref{appTruth}.

Now we show that the canonical model has properties (Func), (Inv-p), (Inv-n), (Pr1) and (Pr2); that is, it is a canonical model for the extended model considered in the Subsection \ref{extended sec}.  
The proofs of the following propositions are available in Appendix \ref{appMain}.

\begin{prop}
\label{s=s'}

For any $w$ in $\mathcal{M}^{C}$, if $R^{\phi C}ww'$, then for all $p \in Prop$: 

\begin{center}

$p \in w \Leftrightarrow  p \in w'$;

and

$\neg p \in w \Leftrightarrow \neg p \in w'$. 

\end{center}

\end{prop}

%
%
%
%
%
%

\begin{prop}
\label{at most 1}

For any $w$ in $\mathcal{M}^{C}$, $w$ has at most one $\phi$-successor.

\end{prop}

\begin{prop}
\label{a unique succ}

For any formula $\phi$: if $\{\phi\} \cup \{\chi \mid \chi \wedge \neg I\chi \in w\}$ is consistent, where $w$ is a consistent prime theory, then $w$ must have a unique $\phi$-successor $w'$, such that $w' = \{\psi \mid [\phi]\psi \in w\}$ and there exists $w'' = \{\phi\} \cup \{\chi \mid \neg I \chi \wedge \chi \in w\}$ such that if $I \chi_{1} \in w'$ then $\chi_{1} \not \in w''$. If $\{\phi\} \cup \{ \chi \mid \chi \wedge \neg I \chi \in w\}$ is inconsistent, then $w$ does not have any $\phi$-successor.

\end{prop}

In light of Theorem \ref{M and M+ eqv}, Prop. \ref{a unique succ} also clarifies the consistency beteween  the (Func) property and the construction of the updated model given in Def. \ref{upd-1}.

\begin{prop}
\label{Prop18}

 If $R^{\phi C}ww'$, then (1) for all $w''$ if $R^{C}ww''$ then $R^{C}w'w''$, (2) there exists $w^{*} = \{\phi\} \cup \{\chi \mid \neg I \chi \wedge \chi \in w\}$ such that $R^{C}w'w^{*}$ and (3) for all $w''$ if $R^{C}ww''$ then $R^{C}w^{*}w''$.

\end{prop}

\begin{prop}
\label{Prop19}

If $R^{\phi C}ww'$ and $R^{C}w'w''$, then either $w'' \subseteq \{\phi\} \cup \{\chi \mid \neg I \chi \wedge \chi \in w\}$ or $w'' \subseteq w'''$ for some $w'''$ s.t. $R^{C}ww'''$.

\end{prop}

Having the canonical model, the completeness proof with respect to extended models satisfying the properties (Func), (Inv-p), (Inv-n), (Pr1), and (Pr2) is straightforward.

\begin{theor}

The system \textbf{LEI$^{up}$} is sound and complete with respect to the extended models satisfying the properties (Func), (Inv-p), (Inv-n), (Pr1), and (Pr2).

\end{theor}

Then, by Theorem \ref{M and M+ eqv}, we have completeness with respect to the semantics described in Definitions \ref{frames} and \ref{upd-1}.

\begin{theor}

\textbf{LEI$^{up}$}  is sound and strongly complete with respect to the semantics of \textbf{LEI$^{up}$} on the class of all Kripke frames.

\end{theor}

\subsection{Modelling (non-)excusable ignorance}

We are now able to model the two examples taken from \citet{Peels2014} that were presented in the introduction. Recall that, in the first situation, an agent is in a position in which he gives a poisoned chocolate bar to his daughter, but his ignorance of the fact that the chocolate might be poisoned provides him with a full excuse for his action. We argued that this situation can be modelled by using the $I$ operator. Let $p$ stand for the proposition `The chocolate bar might be poisoned.'\footnote{Note that to deal with the aletic modality of possibility contained in this proposition, our setting can be extended with this kind of modality. However, to keep the example as simple as possible, we will leave this issue aside and consider the proposition as atomic.} Consider a \textbf{LEI}-model $\mathcal{M}^{e} = (W, R, v)$, with a world $w_{0}$, such that $v(p) = \{w_{0}\}$ and for all $w_{i}$ such that $Rw_{0}w_{i}$, $p$ is not true in $w_{i}$ (see Figure \ref{ex model}). Clearly, in a model so defined,  $\mathcal{M}^{e}, w_{0} \models I p$ holds. It should be noted that $\neg p$ can belong (or not) to some $w_{i}$ that is not $w_{0}$ in this model. If it belongs to $w_{i}$, then this means that the agent considers a possible world in which the chocolate bar might not be poisoned. If neither $p$ nor $\neg p$ belong to $w_{i}$, then this means that the agent does not consider the fact that the chocolate bar might be poisoned. Both possibilities provide either a case of disbelieving ignorance, or of deep ignorance, thus providing a full excuse to the agent's actions.

Now, let us integrate the second example from the introduction into the situation depicted in $\mathcal{M}^{e}$. In this case, the agent hears on the news that some maniac is poisoning chocolate bars in the neighborhood and that the chocolate bar in his and other houses might be poisoned. For simplicity's sake, let us consider only one part of the news that the agent hears, namely the fact that the chocolate bar might be poisoned (i.e., proposition $p$). This news can be considered as an announcement. Thus, whenever $p$ is consistent with all that is true and not ignored by the agent: if $p$ is announced, then one obtains the model $\mathcal{M}^{e}\vert^{w_{0}}_{p}$ (see Figure \ref{model ex}), in which $\mathcal{M}^{e}\vert^{w_{0}}_{p}, w_{0} \not \models I p$. This means that the agent takes the news represented by the announcement into his consideration. However, he does not revise his previous considerations about the chocolate bar in his house because it would be the case if one represents this news via the standard $[! \,\,]$ procedure. Moreover, the news could be not true, as sometimes happens in the real world. Should the fact that the news is not true influence the agent's culpability? We argue that our model provides an answer to this question. If the news is not true but does not contradict any truth of which the agent is not ignorant, then the agent has no reason to disbelieve them. Thus, he has to consider $p$ as true, which (in turn) means that he should be blameworthy for giving the chocolate bar to his daughter. However, if the announcement $p$ clearly contradicts some true statement of which the agent is not ignorant, then he does not take the news seriously and does not add this information to what he considers to be true. In this case, the agent seems to disbelieve $p$, and thus continues to be ignorant of $p$ in the sense of $I$. This is represented by the fact that if $Cn(p) \cup \{\chi \mid \mathcal{M}, w_{0} \models \chi \wedge \neg I \chi\}$ is inconsistent, then the model $\mathcal{M}^{e}\vert^{w_{0}}_{p}$ cannot be constructed and in this case $\mathcal{M}^{e}, w_{0} \models [p]I p$ would hold.

The example presented here can take a more complex form if one considers that there was a proposition $r$ corresponding to `the chocolate bar is open' and that the agent was not ignorant of its truth in the model $\mathcal{M}^{e}$. In this case, one may consider that $p$ is a mere consequence of  $p \wedge r$. This situation can also be represented formally, thus providing a more fine-grained analysis of the conditions under which the agent is morally culpable. To simplify the presentation, we leave this discussion aside but we remark that \textbf{LEI$^{up}$} seems to be a  very natural and intuitively clear setting for representing real world scenarios involving reasoning on the basis of one's ignorance.

\begin{figure}
\centering
\begin{tikzpicture}[modal,world/.append style={minimum size=0.8cm}]
\node[real world] (w) [label=right:$w_{0}$] {$p$};
\node[world1] (s) [label=right:$$, above=of w] {$...$};
\path[->] (w) edge (s);
\end{tikzpicture}
\caption{Model $\mathcal{M}^{e}$}
\label{ex model}
\end{figure}
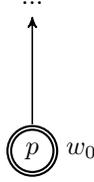

\begin{figure}
\centering
\begin{tikzpicture}[modal,world/.append style={minimum size=0.8cm}]
\node[real world] (w) [label=right:$w_{0}$] {$p$};
\node[world1] (s) [label=right:$$, above=of w] {$...$};
\node[world] (u) [label=right:$w^{w_{0}}_{0p}$, left=of s] {$p$};
\path[->] (w) edge (u);
\path[->] (w) edge (s);
\end{tikzpicture}
\caption{Model $\mathcal{M}^{e}|^{w_{0}}_{p}$}
\label{model ex}
\end{figure}
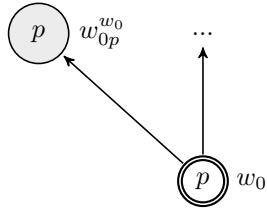

\section{Conclusion}

On the basis of recent debates in the epistemology of ignorance, we endorse the position that the fully excusable ignorance of a proposition $p$ is the one in which an agent does not consider $p$ as true, and is thus unable to act as if this proposition were true. We have introduced an original complete and sound system \textbf{LEI} to model the kind of ignorance that provides an agent with a moral excuse (i.e., disbelieving and deep ignorance). The originality of this framework lies in the fact that it is characterized by Kripke semantics with possibly incomplete worlds. This allowed us to model in a natural and intuitive way situations in which an agent can be not only disbelievingly ignorant but also deeply ignorant. Moreover, to take into account the conditions of a possible change in an agent's ignorance, we extended this system with a public announcement operator, and introduced a complete and sound system \textbf{LEI$^{up}$}. Interestingly, the update procedure that we defined does not  necessarily require the announcement to be true, as  is the case with the usual eliminative public announcement. This dynamic framework permits us to model the transformation of excusable ignorance into non-excusable one, thus clarifying the logical conditions under which an agent can be deemed as morally blameworthy for their actions. Finally, we remark that the update procedure that we have introduced seems to have the potential to be applicable to other different epistemic or doxastic frameworks, thus deserving further investigation.

%
%
%
%
%
%
%
%
%
%
%
%
%
%

\appendix

\section{Appendix}

\subsection{Proof of Proposition \ref{dedt}}
\label{appDT}

\textbf{Proposition  \ref{dedt}.}

\begin{itemize}

\item[($DT$)] $\Gamma, \phi \vdash \psi$  iff $ \Gamma \vdash \phi \rightarrow \psi$

\end{itemize}

\begin{proof}

The direction from right to left is trivial. The direction from left to right can be proven in a standard way. Suppose that $\phi_{1}, ..., \phi_{n}$ is a derivation of $\psi$ from $\Gamma, \phi$. This means that $\phi_{n}$ is $\psi$ and that for each $\phi_{i}$ is either $\phi$, or is in $\Gamma$, or is an axiom, or is inferred by one of the rules of \textbf{LEI}. It is straightforward to prove by induction on $i$ that $\Gamma \vdash \phi \rightarrow \phi_{i}$ for each $\phi_{i}$. We provide only the case of ($IR$), leaving the other proofs for the reader.

Let $\phi_{i}$ be obtained by ($IR$); that is, it is of a form $\chi_{1} \rightarrow (I \chi_{2} \rightarrow I \chi_{1})$ and there is the following step in the derivation resulting in obtaining $\phi_{i}$:

\begin{itemize}

\item[(a)] $\vdash \chi_{1} \rightarrow \chi_{2}$

\item[(b)] $\vdash \chi_{1} \rightarrow (I \chi_{2} \rightarrow I \chi_{1})$ (from (a) by (IR))

\end{itemize}

Then, we can reason as follows:

\begin{enumerate}

\item $\Gamma \vdash \phi \rightarrow (\chi_{1} \rightarrow \chi_{2})$ (induction hypothesis)


\item $\Gamma \vdash \chi_{1} \rightarrow (I \chi_{2} \rightarrow I \chi_{1})$ (from (b), because if $\vdash \alpha$, then $\Gamma \vdash \alpha$ for any $\Gamma$)

\item $\Gamma \vdash (\chi_{1} \rightarrow (I \chi_{2} \rightarrow I \chi_{1})) \rightarrow ((\chi_{1} \rightarrow \chi_{2}) \rightarrow (\chi_{1} \rightarrow (I \chi_{2} \rightarrow I \chi_{1})))$ (axiom scheme $A12$)

\item $\Gamma \vdash (\chi_{1} \rightarrow \chi_{2}) \rightarrow (\chi_{1} \rightarrow (I \chi_{2} \rightarrow I \chi_{1}))$ (from ($MP$), 2, 3)

\item $\Gamma \vdash \phi \rightarrow (\chi_{1} \rightarrow (I \chi_{2} \rightarrow I \chi_{1}))$ (by ($Trans$), 1,4)

\end{enumerate}
\end{proof}

\subsection{Proofs of Propositions \ref{Arrow need-1}, \ref{Arrow need}, \ref{prop5}.}
\label{Fanprops}

\textbf{Proposition  \ref{Arrow need-1}.}

\begin{itemize}

\item[(T1)] $\neg \phi \rightarrow (\phi \rightarrow \psi)$

\end{itemize}

\begin{proof}

\
	
	\begin{enumerate}
			
			\item $\phi \wedge \neg \phi \vdash \psi$ (ECQ)
			
			\item $\neg \phi, \phi \vdash \psi$ (since, by (Adj) we would have the premise of step 1)
			
			\item $\vdash \neg \phi \rightarrow (\phi \rightarrow \psi)$ (from 2, by DT twice)
			
	\end{enumerate}

\end{proof}


\textbf{Proposition  \ref{Arrow need}.}

\begin{itemize}

\item[(T2)] $(\neg \phi \vee \psi) \rightarrow (\phi \rightarrow \psi)$

\end{itemize}

\begin{proof}

\
	
	\begin{enumerate}
			
			\item $\neg \phi \rightarrow (\phi \rightarrow \psi)$ (T1)
			
			\item $\psi \rightarrow (\phi \rightarrow \psi)$ (A12)
			
			\item $(\neg \phi \rightarrow (\phi \rightarrow \psi)) \wedge (\psi \rightarrow (\phi \rightarrow \psi))$ (from 1, 2 by (Adj))
			
			\item $((\neg \phi \rightarrow (\phi \rightarrow \psi)) \wedge (\psi \rightarrow (\phi \rightarrow \psi))) \rightarrow ((\neg \phi \vee \psi) \rightarrow (\phi \rightarrow \psi))  $ (instance of (A6))
			
			\item $(\neg \phi \vee \psi) \rightarrow (\phi \rightarrow \psi)$ (from 3, 4 and MP)
			
	\end{enumerate}

\end{proof}


\textbf{Proposition  \ref{prop5}.}

\begin{itemize}

\item[(R1)] from $\phi \rightarrow (\psi \rightarrow \chi)$ infer $(\phi \wedge \psi) \rightarrow \chi$

\end{itemize}

\begin{proof}

\
	
	\begin{enumerate}
			
			\item $\phi \rightarrow (\psi \rightarrow \chi)$ (assumption)
			
			\item $(\phi \wedge \psi) \rightarrow \phi$ (A1)
			
			\item $(\phi \wedge \psi) \rightarrow (\psi \rightarrow \chi)$ (from 1,2, and Trans)
			
			\item $(\phi \wedge \psi) \rightarrow \psi$ (A2)
			
			\item $(((\phi \wedge \psi) \rightarrow (\psi \rightarrow \chi)) \wedge ((\phi \wedge \psi) \rightarrow \psi)) \rightarrow ((\phi \wedge \psi)\rightarrow ((\psi \rightarrow \chi) \wedge \psi) ) $ (instance of (A3))
			
			\item $ ((\phi \wedge \psi) \rightarrow (\psi \rightarrow \chi)) \wedge ((\phi \wedge \psi) \rightarrow \psi)  $ (from 3,4 by Adj)
			
			\item $(\phi \wedge \psi)\rightarrow ((\psi \rightarrow \chi) \wedge \psi) $ (from 5,6 by MP)
			
			\item $ ((\psi \rightarrow \chi) \wedge \psi) \rightarrow \chi $ (A11)
			
			\item $(\phi \wedge \psi) \rightarrow \chi $ (from 7,8 by Trans)
			
	\end{enumerate}

\end{proof}

\subsection{Proof of Lemma \ref{sound}}
\label{AppA}

\textbf{Lemma  \ref{sound}.} The system \textbf{LEI} is sound with respect to the class of all frames.

\begin{proof}

The case for propositional principles is granted by the soundness result for \textbf{Lt2}  provided in \citep[Theorem 8.6]{Robles2019}. The cases of (\textit{fact}) and $(I \wedge)$ are the same as the cases $(Ie)$ and $(I \wedge)$ in the proof of Lemma 1 in \citep{Kubyshkina2019}. The following proofs rely on Def. \ref{frames}. In light of the equivalence of Def. \ref{frames1} and \ref{frames}, the reader can reconstruct the soundness proofs also in terms of valuational models as in Def. \ref{frames1}.

For ($emI$), assume that for some $(\mathcal{M}, w)$ $\mathcal{M}, w \not \models I \phi \vee \neg I \phi$. This means that (i) $\mathcal{M}, w \not \models I \phi$ and (ii) $\mathcal{M}, w \not \models \neg I \phi$. From (i) it follows that  (iii) there exists a world $w'$ that is not $w$ such that $Rww'$ and $\mathcal{M}, w' \models \phi$, or (iv) $\mathcal{M}, w \not \models \phi$. From (ii) we have (v)  for all $w''$ that are not $w$, if $Rww''$ then $\mathcal{M}, w'' \not \models \phi$, which contradicts (iii), and (vi) $\mathcal{M}, w \models \phi$, which contradicts (iv).

For $(IR)$, assume that $\mathcal{M} \models \phi \rightarrow \psi$; that is, (i) for all $w$ if $\mathcal{M}, w \models \phi$ then $\mathcal{M}, w \models \psi$, (ii) there is a world $w$ such that $\mathcal{M}, w \models \phi$, but (iii) $\mathcal{M}, w \not \models I \psi \rightarrow I \phi$.\footnote{The proof relies on Def.\ref{frames}, where the implicational cases rest on the observation that, although \textit{modus tollens} is not a valid principle in our framework, it does hold whenever the consequent $\psi$ in the assumption $\phi \rightarrow \psi$ is not true, that is, in other terms, $\mathcal{M}, w \not \models \psi$.} From (iii) we obtain (iv) $\mathcal{M}, w \models I \psi$ and (v) $\mathcal{M}, w \not \models I \phi$. From (iv), we have (vi) for all $w'$ such that $w'$ is not $w$, if $Rww'$, then $\mathcal{M}, w' \not \models \psi$. From (vi) and (i) we obtain (vii) for all $w'$ such that $w'$ is not $w$, if $Rww'$, then $\mathcal{M}, w' \not \models \phi$. From (v) we obtain that either (viii) there exists $w''$ such that $Rww''$ and $\mathcal{M}, w''\models \phi$, which contradicts (vii), or (ix) $\mathcal{M}, w \not \models \phi$, which contradicts (ii).

\end{proof}

\subsection{Proof of Lemma \ref{Ext}}
\label{appLemmaExt}

\textbf{Lemma \ref{Ext}.} Let $\phi \not \vdash \psi$, then there exists a consistent prime \textbf{LEI}-theory $\mathcal{T}$ such that $\phi \in \mathcal{T}$ and $\psi \not \in \mathcal{T}$.

\begin{proof}

We enumerate sentences $\phi_{1}, \phi_{2},...$ and then build up a series of theories starting with $\mathcal{T}_{0} = \{\psi': \phi \vdash \psi'\}$. $\mathcal{T}_{n+1}$ is obtained from $\mathcal{T}_{n}$ by adding $\phi_{n+1}$ if one can do so while closing the result under conjunction and implication without thereby getting $\psi$. $\mathcal{T}$ is obtained as the union of all the $\mathcal{T}_{n}$'s, and it is easy to see that it s closed under the principles of \textbf{LEI}. Thus, there exists a \textbf{LEI}-theory $\mathcal{T}$ such that $\phi \in \mathcal{T}$ and $\psi \not \in \mathcal{T}$. Now we show that $\mathcal{T}$ satisfies the property: if $\chi_{1} \in \mathcal{T}$ or $\chi_{2} \in \mathcal{T}$, then $\chi_{1} \vee \chi_{2} \in \mathcal{T}$. Assume that either $\chi_{1} \in \mathcal{T}$ or $\chi_{2} \in \mathcal{T}$. In both cases, by $(A4)$ or by $(A5)$ respectively, we have $\chi_{1} \vee \chi_{2} \in \mathcal{T}$. Thus, $\mathcal{T}$ is a \textbf{LEI}-theory.

To show that $\mathcal{T}$ is prime, assume that it is not, i.e., $\chi_{1} \vee \chi_{2} \in \mathcal{T}$ but $\chi_{1} \not \in \mathcal{T}$ and $\chi_{2} \not \in \mathcal{T}$. Then the theories obtained from $\mathcal{T} \cup \{\chi_{1}\}$ and $\mathcal{T} \cup \{ \chi_{2}\}$ must both contain $\psi$. It follows that there is a conjunction of members of $\mathcal{T}$, $\tau$ such that $\tau \wedge \chi_{1} \vdash \psi$ and $\tau \wedge \chi_{2} \vdash \psi$. By Deduction theorem, this means that $\vdash (\tau \wedge \chi_{1}) \rightarrow \psi$ and $\vdash (\tau \wedge \chi_{2}) \rightarrow \psi$. Then, by $(Adj)$ we get $\vdash (( \tau \wedge \chi_{1}) \rightarrow \psi) \wedge ((\tau \wedge \chi_{2}) \rightarrow \psi)$, and then, by $(A6)$, $\vdash ((\tau \wedge \chi_{1}) \vee (\tau \wedge \chi_{2})) \rightarrow \psi$. We have $\vdash (\tau \wedge (\chi_{1} \vee \chi_{2})) \rightarrow ((\tau \wedge \chi_{1}) \vee (\tau \wedge \chi_{2}))$ (Ax. scheme $(A7)$). Thus, by $(Trans)$ we have $\vdash (\tau \wedge (\chi_{1} \vee \chi_{2})) \rightarrow \psi$. Having in mind that $\tau \wedge (\chi_{1} \vee \chi_{2}) \in \mathcal{T}$, this means that $\psi \in \mathcal{T}$, which contradicts the construction of $\mathcal{T}$.

To show that $\mathcal{T}$ is consistent, assume that it is not. Then, there is $\chi$ such that both $\chi$ and $\neg \chi \in \mathcal{T}$. Then, by $(ECQ)$, $\psi \in \mathcal{T}$ which contradicts the construction of $\mathcal{T}$.

\end{proof}

\subsection{Proof of Lemma \ref{Truth}}
\label{appTruth1}

\textbf{Lemma \ref{Truth}.} For all formulas $\phi$, and all consistent prime \textbf{LEI}-theories $w$, 

\begin{center}

$\mathcal{M}^{C}, w \models \phi$ iff $\phi \in w$;

$\mathcal{M}^{C}, w \models \neg \phi$ iff $\neg \phi \in w$.

\end{center}

\begin{proof}

We prove the lemma by induction on the structure of $\phi$. Notice, that the case of $\neg \psi$ is treated inductively, similarly to semantic clauses for $\neg$ introduced in Definition \ref{frames}.

\textbf{Base case.} By definition of $v^{C}$, $p \in w$, s.t. $w \in W^{C}$ iff $w \in v^{C}(p)$, which is by semantics equivalent to $\mathcal{M}^{C}, w \models p$. In case of $\neg p$ the proof is the same.

\textbf{Induction step.}

\begin{enumerate}

\item $\neg \neg \psi \in w$ iff $\psi \in w$ (by $(A10)$) iff $\mathcal{M}^{C}, w \models \psi$ (induction hypothesis) iff $\mathcal{M}^{C}, w \models \neg \neg \psi$ (by Def. \ref{frames}).

\item $\psi \wedge \chi \in w$ iff $\psi \in w$ and $\chi \in w$ (because $w$ is a theory) iff $\mathcal{M}^{C}, w \models \psi$ and $\mathcal{M}^{C}, w \models \chi$ (induction hypothesis) iff $\mathcal{M}^{C}, w \models \psi \wedge \chi$ (by Def. \ref{frames}).

\item $\neg(\psi \wedge \chi) \in w$ iff  $\neg \psi \vee \neg \chi \in w$ (by $(A9)$) iff $\neg \psi \in w$ or $\neg \chi \in w$ (because $w$ is a prime theory) iff $\mathcal{M}^{C}, w \models \neg \psi$ or $\mathcal{M}^{C}, w \models \neg \chi$ (induction hypothesis) iff $\mathcal{M}^{C}, w \models \neg (\psi \wedge \chi)$ (by Def. \ref{frames}).

\item $\psi \rightarrow \chi \in w$ iff $\psi \in w$ implies $\chi \in w$ (because $w$ is a \textbf{LEI}-theory satisfying ($MP$)) iff $\mathcal{M}^{C}, w \models \psi$ implies $\mathcal{M}^{C}, w \models \chi$ (induction hypothesis) iff $\mathcal{M}^{C}, w \models \psi \rightarrow \chi$ (by Def. \ref{frames}).

\item ($\Rightarrow$) $\neg (\psi \rightarrow \chi) \in w$ implies that $\psi \rightarrow \chi \not \in w$ (because $w$ is consistent). This means that $\psi \in w$ and $\chi \not \in w$, because $w$ is a \textbf{LEI}-theory satisfying ($MP$). By induction hypothesis, this is equivalent to say that $\mathcal{M}^{C}, w \models \psi$ and $\mathcal{M}^{C}, w \not \models \chi$, that is $\mathcal{M}^{C}, w \not \models \psi \rightarrow \chi$, and thus, by Def. \ref{frames}, $\mathcal{M}^{C}, w \models \neg (\psi \rightarrow \chi)$.

\item[] ($\Leftarrow$) $\mathcal{M}^{C}, w \models \neg (\psi \rightarrow \chi)$ implies that $\mathcal{M}^{C}, w \not \models \psi \rightarrow \chi$. This means that $\mathcal{M}^{C}, w \models \psi$ and $\mathcal{M}^{C}, w \not \models \chi$. By induction hypothesis, $\psi \in w$ and $\chi \not \in w$. By $(A14)$, $\psi \rightarrow (\chi \vee \neg (\psi \rightarrow \chi)) \in w$. Thus, $\chi \vee \neg (\psi \rightarrow \chi) \in w$, and then $\neg(\psi \rightarrow \chi) \in w$, because $w$ is a prime theory.

\item Let $\phi = I \psi$. From the consistency of $w$ and from ($emI$), we have that $I \psi \in w$ iff $\neg I \psi \not \in w$. This means that it is sufficient to provide a standard proof of $\mathcal{M}^{C}, w \models I \psi$ iff $I\psi \in w$. 

\begin{itemize}

\item[($\Rightarrow$)] Suppose $I \psi \not \in w$. Because $w$ is a consistent prime theory, either $\psi \in w$, or $\neg \psi \in w$, or $\psi, \neg \psi \not \in w$. In the second and third cases $\psi \not \in w$ and thus, by induction hypothesis, $\mathcal{M}^{C}, w \not \models \psi$. Therefore, $\mathcal{M}^{C}, w \not \models I \psi$.

Let $\psi \in w$, and thus, by induction hypothesis, $\mathcal{M}^{C}, w \models \psi$. Every consistent prime theory can be extended to a maximal consistent prime theory (Lindenbaum's Lemma). Thus, for any $\chi$, such that $\chi, \neg \chi \not \in w$, either $\chi$ or $\neg \chi$ can be added to $w$ with respect to consistency of $w$. Thus, by the definition of $R^{C}$, it is enough to show that the set $\{\psi\} \cup \{\neg \chi | I \chi \in w\}$ is consistent. 

Suppose that the set is inconsistent. Then, there exist $\chi_{1}, ..., \chi_{n}$ such that $\vdash \neg (\neg \chi_{1} \wedge ... \wedge \neg \chi_{n} \wedge  \psi)$. By $(A9)$ this means that $\vdash \neg (\neg \chi_{1} \wedge ... \wedge \neg \chi_{n}) \vee \neg \psi$. By $(T2)$, we have $\vdash \psi \rightarrow \neg (\neg \chi_{1} \wedge ... \wedge \neg \chi_{n})$, that is $\vdash \psi \rightarrow (\chi_{1} \vee ... \vee \chi_{n})$ (by $(A9)$, $(A10)$ and $(Trans)$). From $(IR)$, we obtain $\vdash \psi \rightarrow (I(\chi_{1} \vee ... \vee \chi_{n}) \rightarrow I \psi)$, and thus $I(\chi_{1} \vee ... \vee \chi_{n}) \rightarrow I \psi \in w$ because $\psi \in w$. Having $I \chi_{i} \in w$ for each $\chi_{i}$, we have $(I \chi_{1} \wedge ... \wedge I \chi_{n}) \in w$. From ($I\wedge^{gen}$) we obtain $I(\chi_{1} \vee ... \vee \chi_{n}) \in w$. But this leads to a contradiction: $I \psi \in w$. 


The set $\{\psi\} \cup \{\neg \chi | I \chi \in w\}$ is consistent, and thus it is contained in some maximal consistent \textbf{LEI}-theory $w'$. Suppose that $ \{\neg \chi | I \chi \in w\}$ is non-empty. Thus, there exists some formula $\chi'$ such that $\neg \chi' \in  \{\neg \chi | I \chi \in w\}$ and so $I \chi' \in w$. From (\textit{fact}), $\chi' \in w$, and $\neg \chi' \in w'$. Thus, $w$ and $w'$ are not the same.

In the case of  $\{\neg \chi | I \chi \in w\}$ being empty, we recall that $w$ can be extended to a  maximal consistent \textbf{LEI}-theory. Let us denote the extension of $w$ to the maximal consistent theory by $w_{m}$. Then, there exists at least one propositional variable $p$ that does not occur in $\psi$ such that either $p \in w_{m}$ or $\neg p \in w_{m}$. If $p\in w_{m}$, then $\psi \wedge p \in w_{m}$, and thus $\{\psi \wedge p\}$ is consistent. However, this means that $\{\psi \wedge \neg p\}$ is also consistent: if not, then $\vdash \neg(\psi \wedge \neg p)$ and thus, by ($US$), $\vdash \neg (\psi \wedge \neg\neg p)$ which means that $\{\psi \wedge p\}$ and $w$ are inconsistent. Similarly, one can show that if $\neg p \in w_{m}$, then $\{\psi \wedge \neg p\}$ and $\{\psi \wedge \neg \neg p\}$ are both consistent with $w_{m}$, and thus with $w$. This means that there exists some $w'$ such that $\psi \in w'$ and $w'$ is not the world $w$. By induction hypothesis, we have $\mathcal{M}, w' \models \psi$ and $\mathcal{M}, w \models \psi$ for some $w'$ that is not $w$. Thus, $\mathcal{M}, w \not \models I \psi$.

\item[$(\Leftarrow)$] Let $I \psi \in w$. By definition of $R^{C}$ if $I \psi \in w$, then whenever $R^{C}ww'$, we have $\psi \not \in w'$, and thus $\mathcal{M}^{C}, w' \not \models \psi$. From $I \psi \in w$ we have $\psi \in w$, and thus $\mathcal{M}^{C}, w \models \psi$. This means that for all $w'$ that are not $w$ such that $R^{C}ww'$ we have $\mathcal{M}^{C}, w' \not \models \psi$ and $\mathcal{M}^{C}, w \models \psi$, which gives us $\mathcal{M}^{C}, w \models I \psi$ by definition of $I$.

\end{itemize}

We do not need to consider the case of $\phi = \neg I \psi$ since, in the presence of ($emI$), it is clear that $\mathcal{M}^{C}, w \models I \psi$ iff $\mathcal{M}^{C}, w \not \models \neg I \psi$ and that for any $\psi$ either $I \psi \in w$, or $\neg I \psi \in w$ (because $w$ is a \textbf{LEI}-theory).

\end{enumerate}

\end{proof}

\subsection{Proof of Theorem \ref{SoundUP}}
\label{appSound}

\textbf{Theorem \ref{SoundUP}.} The system \textbf{LEI$^{up}$} is sound with respect to the class of all frames.

\begin{proof}

\

\begin{itemize}

\item[$(AI)$] Assume that (i) $\mathcal{M}, w \models \phi$, and (ii) $\mathcal{M}, w \models I \neg \psi$, and (iii) $\mathcal{M}, w \models [\phi]I \psi$. From (ii), we have (iv) $\mathcal{M}, w \models \neg \psi$ and (v) for all $w'$ s.t. $w'$ is not $w$, if $Rww'$ then $\mathcal{M}, w' \not \models \neg \psi$. From (i) and (iii), it's easy to see that (vi) $\mathcal{M} \vert^{w}_{\phi}, w \models I \psi$, that is (vii) $\mathcal{M} \vert^{w}_{\phi}, w \models \psi$ and (viii) for each $w' \in \mathcal{M}\vert^{w}_{\phi}$ s.t. $w'$ is not $w$, if $Rww'$ then $\mathcal{M}\vert^{w}_{\phi}, w' \not \models \psi$. By the construction of $\mathcal{M}\vert^{w}_{\phi}$ and from (viii), it is obvious that (ix) for each $w' \in \mathcal{M}$ s.t. $w'$ is not $w$, if $Rww'$ then $\mathcal{M}, w' \not \models \psi$. From (v) and (ix) we have (x) for all $w' \in \mathcal{M}$ s.t. $w'$ is not $w$, if $Rww'$ then $\mathcal{M}, w' \not \models \psi \vee \neg \psi$. From (iv), we have (xi) $\mathcal{M}, w \models \psi \vee \neg \psi$, and thus, by (x) and (xi), we have $\mathcal{M}, w \models I (\psi \vee \neg \psi)$.

\item[$(dA\rightarrow)$]  ($\Rightarrow$) Assume that (i) $\mathcal{M}, w \models [\phi](\psi \rightarrow \chi)$ and (ii) $\mathcal{M}, w \models [\phi]\psi$. From (i), we obtain (iii) if $Cn(\phi) \cup \{\gamma \mid \mathcal{M}, w \models \gamma\}$ is consistent, then $\mathcal{M} \vert^{w}_{\phi}, w \models \psi \rightarrow \chi$. Similarly, from (ii), we have that (iv) if $Cn(\phi) \cup \{\gamma \mid \mathcal{M}, w \models \gamma\}$ is consistent, then $\mathcal{M}\vert^{w}_{\phi}, w \models \psi$. From (iii) and (iv), we have that (v) if $Cn(\phi) \cup \{\gamma \mid \gamma \in w\}$ is consistent, then $\mathcal{M}\vert^{w}_{\phi}, w \models \chi$, that is, by definition, $\mathcal{M}, w \models [\phi] \chi$.

\item[] ($\Leftarrow$) Assume that (i) $\mathcal{M}, w \not \models [\phi] (\psi \rightarrow \chi)$. Then, we have (ii) $Cn(\phi) \cup \{\gamma \mid \mathcal{M}, w \models \gamma\}$ is consistent and (iii) $\mathcal{M}\vert^{w}_{\phi}, w \not \models \psi \rightarrow \chi$. This means that (iv) $\mathcal{M}\vert^{w}_{\phi}, w \models \psi$ and (v) $\mathcal{M} \vert^{w}_{\phi}, w \not \models \chi$. From (ii) and (iv), we have (vi). $\mathcal{M}, w \models [\phi] \psi$. From (ii) and (v), we have (vii) $\mathcal{M}, w \not \models [\phi] \chi$. Thus, from (vi) and (vii), we have $\mathcal{M}, w \not \models [\phi]\psi \rightarrow [\phi] \chi$.

\item[$(nI)$] Assume that $\mathcal{M}, w \not \models [\phi] \neg I \phi$. This means that $Cn(\phi) \cup \{ \gamma \mid \mathcal{M}, w \models \gamma\}$ is consistent and $\mathcal{M}\vert^{w}_{\phi}, w \not \models \neg I \phi$, that is $\mathcal{M}\vert^{w}_{\phi}, w  \models I \phi$. This last means that for all $w'$ s.t. $w'$ is not $w$, if $Rww'$ (in $\mathcal{M} \vert^{w}_{\phi}$) then $\mathcal{M} \vert^{w}_{\phi}, w' \not \models \phi$. However, by construction of $\mathcal{M} \vert^{w}_{\phi}$, there exists $w^{w}_{\phi}$ s.t. $Rww^{w}_{\phi}$ and $\mathcal{M} \vert^{w}_{\phi}, w^{w}_{\phi} \models \phi$, which is a contradiction.

\item[$(nA1)$] Assume that (i) $\mathcal{M}, w \models \neg [\phi] \psi$ and (ii) $\mathcal{M}, w \not \models [\phi] (\psi \rightarrow (p \wedge \neg p))$. From (i), we have (iii) $Cn(\phi) \cup \{\gamma \mid \mathcal{M}, w \models \gamma\}$ is consistent and $\mathcal{M}\vert^{w}_{\phi}, w \not \models \psi$. Similarly, from (ii) we have (iv) $Cn(\phi) \cup \{\gamma \mid \mathcal{M}, w \models \gamma\}$ is consistent and $\mathcal{M}\vert^{w}_{\phi}, w \not \models \psi \rightarrow (p \wedge \neg p)$. From (iv) we have that (v) $\mathcal{M} \vert^{w}_{\phi}, w \models \psi$, which contradicts (iii).

\item[$(nA2)$] Assume that $\mathcal{M}, w \models \neg \phi$ and $\mathcal{M}, w \not \models [\phi] (p \wedge \neg p)$. This last means that $Cn(\phi) \cup \{ \gamma \mid \mathcal{M}, w \models \gamma\}$ is consistent, which is not possible because of $\mathcal{M}, w \models \neg \phi$.

\item[$(dA\vee)$] Let $\mathcal{M}, w \models [\phi] (\psi \vee \chi)$ for an arbitrary $w$. This is equivalent to the fact that if $Cn(\phi) \cup \{\gamma \mid \mathcal{M}, w \models \gamma\}$ is consistent, then $\mathcal{M} \vert^{w}_{\phi}, w \models \psi \vee \chi$. This is also equivalent to the fact that  if $Cn(\phi) \cup \{\gamma \mid \mathcal{M}, w \models \gamma\}$ is consistent, then $\mathcal{M} \vert^{w}_{\phi}, w \models \psi$ or $\mathcal{M} \vert^{w}_{\phi}, w \models \chi$. The last one is also equivalent to $\mathcal{M}, w \models [\phi] \psi \vee [\phi] \chi$.

\item[$(emA)$] Assume that (i) $\mathcal{M}, w \not \models [\phi]\psi \vee \neg [\phi] \psi$. This means that (ii) $\mathcal{M}, w \not \models [\phi] \psi$ and (iii) $\mathcal{M}, w \not \models \neg [\phi] \psi$. From (ii), by definition, we have (iv) $\mathcal{M}, w \models \neg [\phi]\psi$, which contradicts to (iii).

\item[$(INV)$] By construction of the new model $\mathcal{M}\vert^{w}_{\phi}$ with respect to a world $\phi \in \mathcal{M}$, the world $w$ in both $\mathcal{M}$ and $\mathcal{M}\vert^{w}_{\phi}$ has the same atomic propositions. The creation of the new world and adding the new accessibility relations impact only the value of modal formulas in $w$, but not the propositional content.

\item[$(nAp1)$] Assume that (i) $\mathcal{M}, w \models \neg [\phi] \neg p$. Then, (ii) $Cn(\phi) \cup \{\gamma \mid \mathcal{M}, w \models \gamma\}$ is consistent and (iii) $\mathcal{M} \vert^{w}_{\phi}, w \not \models \neg p$. The step (iii) admits two possibilities: either (iv) $\mathcal{M} \vert^{w}_{\phi}, w \models p$, or (v) $\mathcal{M}\vert^{w}_{\phi} \not \models p$. If (iv), then we have $\mathcal{M}, w \models p$ (by the same reason that $(INV)$ holds), and thus $\mathcal{M}, w \models p \vee [\phi](p \rightarrow (q \wedge \neg q))$. If (v), then we have, by our truth conditions for implication, (vi) $\mathcal{M} \vert^{w}_{\phi}, w \models p \rightarrow (q \wedge \neg q)$, that is $\mathcal{M}, w \models [\phi](p \rightarrow (q \wedge \neg q))$, and thus $\mathcal{M}, w \models p \vee [\phi](p \rightarrow (q \wedge \neg q))$.

\item[$(nAp2)$] The reasoning is similar to the case of $(nAp1)$.

\item[$(uA)$] Assume that $\mathcal{M}, w \models [\phi] \psi \wedge [\phi] \neg \psi$. This is possible only in case that if $Cn(\phi) \cup \{\gamma \mid \mathcal{M}, w \models \gamma\}$ is consistent then $\mathcal{M}\vert^{w}_{\phi}, w \models \psi$ and $\mathcal{M}\vert^{w}_{\phi}, w \models \neg \psi$. By definitions of operators and of $\mathcal{M}\vert^{w}_{\phi}$, this means that $Cn(\phi) \cup \{\gamma \mid \mathcal{M}, w \models \gamma\}$ is inconsistent, and thus $\mathcal{M}, w \not \models \phi$. Then, by the definition of implication, $\mathcal{M}, w \models \phi \rightarrow (p \wedge \neg p)$.

\item[$(nec)$] Assume that (i) $\mathcal{M} \models \phi$ for all $\mathcal{M}$ (i.e., $\phi$ is valid). Let (ii) $\mathcal{M}, w \not \models [\psi] \phi$ for some formula $\psi$. From (ii), we have (iii) $Cn(\psi) \cup \{\gamma \mid \mathcal{M}, w \models \gamma \}$ is consistent and (iv) $\mathcal{M} \vert^{w}_{\psi}, w \not \models \phi$, which contradicts to (i).

\item[$(intA1)$] Assume that (i) $\mathcal{M} \models \phi \rightarrow \psi$ for all $\mathcal{M}$ (i.e., $\phi \rightarrow \psi$ is valid). Let (ii) $\mathcal{M}, w \not \models [\phi] \neg I \psi$. From (ii) it follows that (iii) $Cn(\phi) \cup \{\gamma \mid \mathcal{M}, w \models \gamma\}$ is consistent and (iv) $\mathcal{M} \vert^{w}_{\phi}, w \not \models \neg I \psi$, that is (v) $\mathcal{M} \vert^{w}_{\phi}, w  \models  I \psi$. From (v), we have (vi) $\mathcal{M} \vert^{w}_{\phi}, w \models \psi$ and (vii) for all $w'$ s.t. $w'$ is not $w$, if $Rww'$ then $\mathcal{M} \vert^{w}_{\phi}, w \not \models \psi$. However, because of (i), we know that $\psi \in Cn(\phi)$, and, by construction of $\mathcal{M}\vert^{w}_{\phi}$, $\psi$ belongs to (at least) one world distinct from $w$ in the updated model, which leads to a contradiction.

\item[$(intA2)$] Assume that (i) $\mathcal{M} \models (\phi \wedge \psi) \rightarrow \chi$ for all $\mathcal{M}$ (i.e., $(\phi \wedge \psi) \rightarrow \chi$ is valid). Also, assume that, for some $w \in \mathcal{M}$, (ii) $\mathcal{M}, w \not \models (\psi \wedge \neg I \psi) \rightarrow [\phi]\neg I \chi$. From (ii), we have (iii) $\mathcal{M}, w \models \psi \wedge \neg I \psi$ and (iv) $\mathcal{M}, w \not \models [\phi]\neg I \chi$. From (iv), applying Definition \ref{upd-1}, it follows that (v)  $Cn(\phi) \cup \{\gamma \mid \mathcal{M}, w \models \gamma \}$ is consistent, and also (vi) $\mathcal{M}\vert^{w}_{\phi}, w \not \models \neg I \chi$. From (vi), we obtain (vii) $\mathcal{M}\vert^{w}_{\phi}, w \models  I \chi$, which entails (viii) $\mathcal{M}\vert^{w}_{\phi}, w \models  \chi$ and (ix) for each $w' \neq w$ s.t. $Rww'$ where $R \in \mathcal{M}\vert^{w}_{\phi}$, $\mathcal{M}\vert^{w}_{\phi}, w' \not \models  \chi$. However, from (i) and (iii), we know that (x) $\chi \in Cn (Cn(\phi) \cup \{\gamma \mid \mathcal{M}, w \models \gamma  \wedge \neg I \gamma \})$, and, by construction of $\mathcal{M}\vert^{w}_{\phi}$, we know that (xi) there is (at least) one world $w'$ s.t. $w' \neq w$ and $Rww'$  where $R \in \mathcal{M}\vert^{w}_{\phi}$, and $\mathcal{M}\vert^{w}_{\phi}, w' \models \chi$, which contradicts (ix).

\item[$(CN)$] Assume that (i) $\mathcal{M}, w \models I \psi \wedge \neg [\phi]I \psi$, where the set of all true propositions in $w$ is $\Gamma$. Let (ii) $\mathcal{M}, w' \not \models \psi$, where the set of all true propositions in $w'$ is $Cn(\phi) \cup \Gamma' = \{\chi \mid \mathcal{M}, w \models \chi  \wedge \neg I \chi\}$. From (i) we have that (iii) for all $w''$ s.t. $w''$ is not $w$ if $Rww''$ then $\mathcal{M}, w'' \not \models \psi$, (iv) $Cn(\phi) \cup \{ \Gamma\}$ is consistent, and (v) $\mathcal{M}\vert^{w}_{\phi}, w \models \neg I \psi$. Consider the world $w'$ described in the condition (ii), it is such that it only validates propositions contained in $Cn(\phi) \cup \{\chi \mid \mathcal{M}, w \models \chi  \wedge \neg I \chi\}$, that is the definition of the world $w^{w}_{\phi} \in \mathcal{M}\vert^{w}_{\phi}$. From (ii) and the definition of $w^{w}_{\phi}$ we have that (vi) $\mathcal{M} \vert^{w}_{\phi}, w^{w}_{\phi} \not \models \psi$. From (v) we have that either (vii) $\mathcal{M}\vert^{w}_{\phi}, w \not \models \psi$, or (viii) there exists $w'''$ s.t. $w'''$ is not $w$, $Rww'''$ (in $\mathcal{M}\vert^{w}_{\phi}$), and $\mathcal{M}\vert^{w}_{\psi}, w''' \models \psi$. By the construction of $\mathcal{M}\vert^{w}_{\phi}$, the case (viii) contradicts (iii) and (vi). The case (vii) is possible only if $w^{w}_{\phi}$ contains new information about the value of $\psi$ (because if $\psi$ does not contain modal operators, its value remains the same, and if it contains modalities, its value depends on its value in the accessible worlds). However, from (i) we have $\mathcal{M}, w \models \psi$, which means that $\neg \psi \not \in Cn(\phi) \cup \{ \chi \mid \mathcal{M}, w \models \chi  \wedge \neg I \chi\}$ (and thus $\mathcal{M} \vert^{w}_{\phi}, w^{w}_{\phi} \not \models \neg \psi$). This last observation, taken together with (vii), indicates that the accessibility of $w^{w}_{\phi}$ from $w$ does not change the value of $\psi$, which contradicts (vii).

\end{itemize}

\end{proof}

\subsection{Proof of Proposition \ref{prop10}}
\label{appProp10}

\textbf{Proposition \ref{prop10}.} The following is a theorem of \textbf{LEI$^{up}$}:

\begin{itemize}

\item[$(dA\wedge)$] $ [\phi](\psi \wedge \chi) \leftrightarrow ([\phi]\psi \wedge [\phi]\chi)$

\end{itemize}

\begin{proof}

\

\begin{enumerate}

\item $(\psi \wedge \chi) \rightarrow \psi$ (from ($A1$))

\item $(\psi \wedge \chi) \rightarrow \chi$ (from ($A2$))

\item $[\phi] ((\psi \wedge \chi) \rightarrow \psi)$ (from 1 by ($nec$))

\item $[\phi] ((\psi \wedge \chi) \rightarrow \psi) \rightarrow ([\phi] (\psi \wedge \chi) \rightarrow [\phi] \psi)$ (from ($dA\rightarrow$))

\item $[\phi] (\psi \wedge \chi) \rightarrow [\phi] \psi$ (from 3, 4 by ($MP$))

\item $[\phi] ((\psi \wedge \chi) \rightarrow \chi)$ (from 2 by ($nec$))

\item $[\phi] ((\psi \wedge \chi) \rightarrow \chi) \rightarrow ([\phi] (\psi \wedge \chi) \rightarrow [\phi] \chi)$ (from ($dA\rightarrow$))

\item $[\phi] (\psi \wedge \chi) \rightarrow [\phi] \chi$ (from 6,7 by ($MP$))

\item $[\phi] (\psi \wedge \chi) \rightarrow ([\phi] \psi \wedge [\phi] \chi)$ (from ($A3$), 5, 8 by ($MP$))

\item $\psi \rightarrow (\chi \rightarrow (\psi \wedge \chi))$ (from ($Adj$) by Deduction Theorem)

\item $[\phi](\psi \rightarrow (\chi \rightarrow (\psi \wedge \chi)))$ (from 10 by ($nec$))

\item $[\phi]\psi \rightarrow [\phi] (\chi \rightarrow (\psi \wedge \chi))$ (from ($dA \rightarrow$), 11 by ($MP$))

\item $[\phi] (\chi \rightarrow (\psi \wedge \chi)) \rightarrow ([\phi] \chi \rightarrow [\phi] (\psi \wedge \chi))$ (from ($dA \rightarrow$))

\item $[\phi]\psi \rightarrow ([\phi] \chi \rightarrow [\phi] (\psi \wedge \chi))$ (from 12, 13 by ($Trans$))

\item $([\phi]\psi \wedge [\phi] \chi) \rightarrow [\phi] (\psi \wedge \chi))$ (from 14 by ($R1$))

\item $[\phi] (\psi \wedge \chi) \leftrightarrow ([\phi] \psi \wedge [\phi] \chi)$ (from 9, 15, ($Adj$), definition of $\leftrightarrow$) 

\end{enumerate}

\end{proof}

\subsection{Proof of Proposition \ref{intA2gen}}
\label{appintA2gen}

\textbf{Proposition \ref{intA2gen}}. For all $n \geq 1$:

\begin{itemize}

\item[$(intA2)^{gen}$] From $\vdash (\phi \wedge \psi_{1} \wedge  ... \wedge \psi_{n}) \rightarrow \chi$ infer 

$\vdash (\psi_{1} \wedge \neg I \psi_{1} \wedge ... \wedge \psi_{n} \wedge \neg I \psi_{n}) \rightarrow [\phi] \neg I \chi$.

\end{itemize}

\begin{proof}

We prove the proposition by induction on $n$.

\textbf{Basic step.} From $\vdash (\phi \wedge \psi_{1}) \rightarrow \chi$ infer $\vdash (\psi_{1} \wedge \neg I \psi_{1}) \rightarrow [\phi] \neg I \chi$. This is obtained by an instantiation into the rule $(intA2)$.

\textbf{Inductive step.} Assume by induction hypothesis (IH) that the proposition holds for $n=k$. We show that:

From $\vdash (\phi \wedge \psi_{1} \wedge  ... \wedge \psi_{k+1}) \rightarrow \chi$ infer 

$\vdash (\psi_{1} \wedge \neg I \psi_{1} \wedge ... \wedge \psi_{k+1} \wedge \neg I \psi_{k+1}) \rightarrow [\phi] \neg I \chi$.

\begin{enumerate}

\item $\vdash (\phi \wedge \psi_{1} \wedge  ... \wedge \psi_{k+1}) \rightarrow \chi$ (assumption)

\item From $\vdash (\phi \wedge \psi_{1}  ... \wedge \psi_{k}) \rightarrow \chi$ infer  $\vdash (\psi_{1} \wedge \neg I \psi_{1} \wedge ... \wedge \psi_{k} \wedge \neg I \psi_{k}) \rightarrow [\phi] \neg I \chi$ (IH)

\item From $\psi_{k+1}  \vdash (\phi \wedge \psi_{1}  \wedge ... \wedge \psi_{k}) \rightarrow \chi$ infer  $\psi_{k+1} \vdash (\psi_{1} \wedge \neg I \psi_{1} \wedge ... \wedge \psi_{k} \wedge \neg I \psi_{k}) \rightarrow [\phi] \neg I \chi$ (because whenever from $\vdash A$ infer $\vdash B$, one can add supplementary premises, that is, from $C\vdash A$ infer $C\vdash B$)

\item From $\psi_{k+1}  \vdash (\phi \wedge \psi_{1}  \wedge ... \wedge \psi_{k}) \rightarrow \chi$ infer  $\psi_{k+1} \wedge \neg I \psi_{k+1}\vdash (\psi_{1} \wedge \neg I \psi_{1} \wedge ... \wedge \psi_{k} \wedge \neg I \psi_{k}) \rightarrow [\phi] \neg I \chi$ (because if $A \vdash B$, then $A \wedge C \vdash B$),

\item From $\phi \wedge \psi_{1} \wedge ... \wedge \psi_{k+1} \vdash  \chi$ infer $\psi_{1} \wedge \neg I \psi_{1} \wedge ... \wedge \psi_{k+1} \wedge \neg I \psi_{k+1}\vdash  [\phi] \neg I \chi$ (from $(DT)$ and $(Adj)$ applied to premises)

\item From $\vdash (\phi \wedge \psi_{1}  \wedge ... \wedge \psi_{k+1}) \rightarrow \chi$ infer $\vdash (\psi_{1} \wedge \neg I \psi_{1} \wedge ... \wedge \psi_{k+1} \wedge \neg I \psi_{k+1}) \rightarrow [\phi] \neg I \chi$ (from $(DT)$)

\item $\vdash (\psi_{1} \wedge \neg I \psi_{1} \wedge ... \wedge \psi_{k+1} \wedge \neg I \psi_{k+1}) \rightarrow [\phi] \neg I \chi$ (from 1 and 5)

\end{enumerate}

\end{proof}

\subsection{Proof of Lemma \ref{u-R eqv}}
\label{appL4}

\textbf{Lemma \ref{u-R eqv}.} If for all $w''$ in $\mathcal{M}$: $\mathcal{M}^{-}, w'' \models \phi \Leftrightarrow \mathcal{M}, w'' \models^{+} \phi$, and $R^{\psi}ww'$, then

\begin{center}

$\mathcal{M}^{-}\vert^{w}_{\psi}, w \models \phi$ iff $\mathcal{M},w' \models^{+} \phi$.

\end{center}

\begin{proof}

Let us consider the point $(\mathcal{M}^{-}\vert^{w}_{\psi}, w)$. 
By definition of $\mathcal{M}^{-}\vert^{w}_{\psi}$ for any literal $p$ or $\neg p$ (let us call it $p^{*}$), $\mathcal{M}^{-}\vert^{w}_{\psi}, w \models p^{*}$ iff $\mathcal{M}^{-}, w \models p^{*}$. Thus, in accordance with our assumption that for any $w'$ in $\mathcal{M}$ we have $\mathcal{M}^{-}, w' \models \phi \Leftrightarrow \mathcal{M}, w' \models^{+} \phi$, by (Inv-p) and (Inv-n) we have $\mathcal{M}^{-}\vert^{w}_{\psi}, w  \models p^{*}$ iff $\mathcal{M}, w' \models^{+} p^{*}$. Now let us assure that points $(\mathcal{M}^{-}\vert^{w}_{\psi}, w)$ and $(\mathcal{M}, w')$ access exactly the same information. By definition of $\mathcal{M}^{-}\vert^{w}_{\psi}$ we have that (1) $Rww_{1}$ in $\mathcal{M}^{-}\vert^{w}_{\psi}$ for all $w_{1}$ such that $Rww_{1}$ in $\mathcal{M}^{-}$, (2) there exists $w^{w}_{\psi} = Cn(\psi) \cup \{\chi \mid \mathcal{M}, w \models \chi  \wedge \neg I \chi\}$ where $w \in \mathcal{M}^{-}$ such that for $w \in \mathcal{M}^{-}\vert^{w}_{\psi}$ we have $Rww^{w}_{\psi}$ and (3) for all $w_{2}$ if $Rww_{2}$ in $ \mathcal{M}^{-}$, then $Rw^{w}_{\psi}w_{2}$. In accordance with our assumption that for all $w'$ in $\mathcal{M}$: $\mathcal{M}^{-}, w' \models \phi \Leftrightarrow \mathcal{M}, w' \models^{+} \phi$ and that $R^{\psi}ww'$, by (Pr1) we have the same conditions for $(\mathcal{M}, w')$: (1*) $Rw'w''$ for all $w''$ such that $Rww''$, (2*) there exists $w^{*} = Cn(\psi) \cup \{\chi \mid \mathcal{M}, w \models \chi  \wedge \neg I \chi\}$ such that $Rw'w^{*}$ (notice that $w^{*} = w^{w}_{\psi}$), and (3*) for all $w'''$ if $Rww'''$, then $Rw^{*}w'''$. Thus, the point $(\mathcal{M}, w')$ accesses all the worlds accessible from the point $(\mathcal{M}^{-}\vert^{w}_{\psi}, w)$. The property (Pr2) assures that the point $(\mathcal{M}, w')$ does not access any world which validates a formula which is not contained either in $w^{*}$, or in a world accessible from $w$, that is, the point $(\mathcal{M}, w')$ accesses only information available in the point $(\mathcal{M}^{-}\vert^{w}_{\psi}, w)$. Thus, $\mathcal{M}^{-}\vert^{w}_{\psi}, w \models \phi$ iff $\mathcal{M},w' \models^{+} \phi$.

\end{proof}

\subsection{Proof of Theorem \ref{equiv ext}}
\label{appT12}

\textbf{Theorem \ref{equiv ext}.} For any $w \in W$,

$\mathcal{M}^{-}, w \models \phi$ iff $\mathcal{M}, w \models^{+} \phi$.

\begin{proof}

We prove this by induction on the length of a formula. The cases of atomic propositions, non-modal operators, and $I$ operator are trivial, because the Definitions \ref{frames} and \ref{ext model def sat} for these terms are the same. We thus show only that $\mathcal{M}^{-}, w \models [\phi]\psi$ iff $\mathcal{M}, w \models^{+} [\phi]\psi$.

\begin{itemize}

\item[($\Rightarrow$)] Let $\mathcal{M}, w \not \models^{+} [\phi] \psi$. Then, $R^{\phi}ww'$ and $\mathcal{M}, w' \not \models^{+} \psi$.  By (Func) we have that $Cn(\phi) \cup \{\chi \mid \mathcal{M}, w \models \chi\}$ is consistent. By Lemma \ref{u-R eqv}, taking into account the induction hypothesis, we obtain $\mathcal{M}^{-}\vert^{w}_{\phi}, w \not \models \psi$, which means that $\mathcal{M}^{-}, w \not \models [\phi] \psi$.

\item[($\Leftarrow$)] Let $\mathcal{M}^{-}, w \not \models [\phi] \psi$; that is,  $Cn(\phi) \cup \{\chi \mid \mathcal{M}, w \models \chi\}$ is consistent and $\mathcal{M}^{-}\vert^{w}_{\phi}, w \not \models \psi$. From the consistency of $Cn(\phi) \cup \{\chi \mid \mathcal{M}, w \models \chi\}$ and from (Func), we have $R^{\phi}ww'$ for a unique $w'$. By Lemma \ref{u-R eqv}, and in accordance with the induction hypothesis, we have $\mathcal{M}, w' \not \models^{+} \psi$, which means that $\mathcal{M}, w \not \models^{+} [\phi]\psi$.

\end{itemize}
\end{proof}

\subsection{Proof of Lemma \ref{ext truth lemma}}
\label{appTruth}

\textbf{Lemma \ref{ext truth lemma}.} For all formulas $\phi \in \mathcal{L}^{up}$, and all consistent prime $LEI$-theories $w$,

\begin{center}

$\mathcal{M}^{C}, w \models^{+} \phi \Leftrightarrow \phi \in w$

and

$\mathcal{M}^{C}, w \models^{+} \neg \phi \Leftrightarrow \neg \phi \in w$

\end{center}

\begin{proof}

To simplify the reading of the proof, we omit the superscripts `C' and `+'.

The cases of propositional operators and $I$ operator are the same as in the proof of Lemma \ref{Truth}, except that we replace the use of ($US$) by ($US^{up}$). The case for formulas of a form $[\phi]\psi$ follows.

Bearing in mind that for any $\phi$ and $\psi$,  $\mathcal{M}, w \models [\phi]\psi$ iff $\mathcal{M}, w \not \models \neg [\phi]\psi$ and that $[\phi]\psi \in w$ iff $\neg [\phi]\psi \not \in w$, it is sufficient to show that $\mathcal{M}, w \models [\phi] \psi \Leftrightarrow [\phi]\psi \in w$.

\begin{itemize}

\item[($\Rightarrow$)] Suppose that $[\phi]\psi \not \in w$. We need to show that there exists $w'$ such that $R^{\phi} ww'$ and $\psi \not \in w'$.

As before, every consistent prime theory can be extended to a maximal consistent prime theory (Lindenbaum's Lemma). Thus, for any $\chi$, such that $\chi, \neg \chi \not \in w$, either $\chi$ or $\neg \chi$ can be added with respect to consistency of $w$. Thus, by the definition of $R^{\phi}$, we need to show that the set $\{\neg \psi\} \cup \{\chi \mid [\phi]\chi \in w\}$ is consistent. 

The proof of the consistency of $\{\neg \psi\} \cup \{\chi \mid [\phi]\chi \in w\}$ is straightforward. Suppose that the set is inconsistent, then there exist $\chi_{1}, ..., \chi_{n}$ such that $\vdash (\chi_{1} \wedge ... \wedge \chi_{n}) \rightarrow \psi$. By ($nec$) we have $\vdash [\phi]((\chi_{1} \wedge ... \chi_{n}) \rightarrow \psi)$. By Proposition \ref{gen!} and construction of $\chi_{i}$, we have $[\phi] (\chi_{1} \wedge ... \wedge \chi_{n}) \in w$. By axiom scheme ($dA \rightarrow$) and ($MP$) we have thus $[\phi]\psi \in w$, which contradicts our assumption.

Let us now show that there exists $w'' = \{\phi\} \cup \{\chi_{1} \mid \chi_{1} \wedge \neg I \chi_{1} \in w\}$ such that  $Rw'w''$, where $w'$ is the superset of $\{\neg \psi\} \cup \{\chi \mid [\phi]\chi \in w\}$, that is to show that for any formula $\alpha$ if $I\alpha \in w'$ then $\alpha \not \in w''$. Suppose that $\alpha \in w''$. Thus, $\alpha$ is a consequence of $\phi \wedge \chi_{1}$ for some $\chi_{1} \in w''$; that is, $\vdash (\phi \wedge \chi_{1}) \rightarrow \alpha$.  Let $\chi_{1} \wedge \neg I \chi_{1} \in w$. By rule ($intA2$), we have $[\phi] \neg I \alpha \in w$, which implies $I \alpha \not \in w'$. If there is no $\chi_{1}$ such that $\chi_{1} \wedge \neg I \chi_{1} \in w$, then, by construction of $w''$, $\alpha$ is a consequence of $\phi$. Thus, by $(intA1)$, $[\phi] \neg I \alpha \in w$, which implies $I \alpha \not \in w'$.

Thus, we have constructed a point $w'$, such that $R^{\phi}ww'$ and $\neg \psi \in w'$, which means that $\mathcal{M}, w \not \models [\phi]\psi$.

\item[($\Leftarrow$)] Let $[\phi]\psi \in w$. By definition of $R^{\phi}$, for all $\phi$ and $\psi$: if $[\phi]\psi \in w$, then ($\psi \in w'$ and there exists $w''$ such that $Rw'w''$ and $w'' = \{\phi \} \cup \{ \chi_{1} \mid \chi_{1} \wedge \neg I \chi_{1} \in w\}$). By induction hypothesis, this means that for all $w'$ if $R^{\phi}ww'$ then $\mathcal{M}, w' \models \psi$; that is, $\mathcal{M}, w \models [\phi]\psi$.

\end{itemize}
\end{proof}

\subsection{Proofs of Propositions \ref{s=s'} --  \ref{Prop19}.}
\label{appMain}

\textbf{Proposition \ref{s=s'}}. For any $w$ in $\mathcal{M}^{C}$, if $R^{\phi C}ww'$, then for all $p \in Prop$: 

\begin{center}

$p \in w \Leftrightarrow  p \in w'$;

and

$\neg p \in w \Leftrightarrow \neg p \in w'$. 

\end{center}

\begin{proof}

Let $p \in w$. Then, by $(INV)$ we have $[\phi] p \in w$, and thus, by definition of $R^{\phi C}$, $p \in w'$. The case of  $\neg p \in w$ is similar.

Let $p \in w'$. This means that $\neg p \not \in w'$, that is, by definition of $R^{\phi C}$, $\neg [ \phi ] \neg p \in w$. By $(nAp1)$, we get $p \vee [\phi](p \rightarrow (q \wedge \neg q)) \in w$. Then, either $p \in w$, or $[\phi](p \rightarrow (q \wedge \neg q)) \in w$. In the second case, $p \rightarrow (q \wedge \neg q) \in w'$. By our assumption that $p \in w'$, we get $q \wedge \neg q \in w'$ which contradicts the consistency of $w'$. Thus, $p \in w$. The case of $\neg p \in w'$ can be obtained similarly by using ($nAp2$).

\end{proof}

\textbf{Proposition \ref{at most 1}}. For any $w$ in $\mathcal{M}^{C}$, $w$ has at most one $\phi$-successor.

\begin{proof}

Suppose that $w$ has two different $\phi$-successors $w'$ and $w''$. Because $w'$ and $w''$ are different, then there exists a $\psi$ such that $\psi \in w'$ and $\psi \not \in w''$, then, by definition of $R^{\phi C}$ we have $\neg [\phi] \psi \in w$. By axiom scheme ($nA1$) we have thus $[\phi](\psi \rightarrow (p \wedge \neg p)) \in w$. From $R^{\phi C} ww'$ we have $\psi \rightarrow (p \wedge \neg p) \in w'$, from which by ($MP$) we obtain $p \wedge \neg p \in w'$, which is a contradiction. Thus, for any $w$ in $\mathcal{M}^{C}$, $w$ has at most one $\phi$-successor.

\end{proof}

\textbf{Proposition \ref{a unique succ}.} For any formula $\phi$: if $\{\phi\} \cup \{\chi \mid \chi  \in w\}$ is consistent, where $w$ is a consistent prime theory, then $w$ must have a unique $\phi$-successor $w'$, such that $w' = \{\psi \mid [\phi]\psi \in w\}$ and there exists $w'' = \{\phi\} \cup \{\chi \mid \neg I \chi \wedge \chi \in w\}$ such that if $I \chi_{1} \in w'$ then $\chi_{1} \not \in w''$. If $\{\phi\} \cup \{ \chi \mid \chi  \in w\}$ is inconsistent, then $w$ does not have any $\phi$-successor.

\begin{proof}

First, we show that if $\{\phi\} \cup \{\chi \mid \chi \in w\}$ is consistent then $w$ must have a $\phi$-successor $w'$, such that $w' = \{\psi \mid [\phi]\psi \in w\}$; that is, $w' = \{\psi \mid [\phi]\psi \in w\}$ is a consistent and prime theory. Clearly, by construction of $w'$, it is closed under $\vdash$ and it is a theory.  Suppose that $w'$ is inconsistent, then there is a formula $\chi_{1}$ such that $\chi_{1} \wedge \neg \chi_{1} \in w'$. By construction of $w'$, this means that $[\phi] (\chi_{1} \wedge \neg \chi_{1}) \in w$. By $(dA\wedge)$ we have $[\phi] \chi_{1} \wedge [\phi]\neg \chi_{1} \in w$. By $(uA)$ this means that $\phi \rightarrow (p \wedge \neg p) \in w$, which contradicts the assumption that $\{\phi\} \cup \{\chi \mid \chi \in w\}$ is consistent.
To show that $w'$ is prime, suppose that $\chi_{1} \vee \chi_{2} \in w'$, and thus $[\phi] (\chi_{1} \vee \chi_{2}) \in w$. From this, by ($dA\vee$) we get that either $[\phi]\chi_{1} \in w$ or $[\phi]\chi_{2} \in w$; that is, either $\chi_{1} \in w'$ or $\chi_{2} \in w'$. 

Now we show that there exists $w'' = \{\phi\} \cup \{\chi \mid \neg I \chi \wedge \chi \in w\}$ such that if $I \chi_{1} \in w'$ then $\chi_{1} \not \in w''$. Clearly, by construction of $w''$, it is closed under $\vdash$ and it is a theory. The consistency of $w''$ is assured by our assumption that $\{\phi\} \cup \{\chi \mid \chi  \in w\}$ is consistent. To show that $w''$ is prime,  let us consider $w$ which is a prime consistent theory, and consider a subset of $w$ which contains only $\chi_{1}, ..., \chi_{n}$, where $\chi_{1}, ..., \chi_{n} \in \{\chi \mid \neg I \chi \wedge \chi \in w\}$. This subset can be extended with $\phi$ (because $\{\phi\} \cup \{\chi \mid \chi \in w\}$ is consistent), and thus it is a consistent prime theory $w^{-}$, which coincides by definition with $w''$. Thus, $w''$ is also prime.


To prove that if $I \chi_{1} \in w'$ then $\chi_{1} \not \in w''$, assume that it is not the case, that is, (i) $I \chi_{1} \in w'$ and (ii) $\chi_{1} \in w''$. By construction of $w''$ (ii) means that (iii) $\phi \wedge \chi^{1} \wedge ... \wedge \chi^{n} \vdash \chi_{1}$ where all $\chi^{1}, ..., \chi^{n}$ are such that $\chi^{1} \wedge \neg I \chi^{1} \in w$, ...,$\chi^{n} \wedge \neg I \chi^{n} \in w$. From (i), by construction of $w'$, we have (iv) $[\phi] I \chi_{1} \in w$. 
From (iii), by Deduction Theorem, we have (iv) $\vdash (\phi \wedge \chi^{1} \wedge ... \wedge \chi^{n}) \rightarrow \chi_{1}$.
By $(intA2^{gen})$, this means that (v) $\vdash (\chi^{1} \wedge \neg I \chi^{1} \wedge ... \wedge \chi^{n} \wedge \neg I \chi^{n}) \rightarrow [\phi] \neg I \chi_{1}$. From (v), by construction of $w''$, we have $[\phi] \neg I \chi_{1} \in w$. Thus, by construction of $w'$, $\neg I \chi_{1} \in w'$ which contradicts (i).

Now we show that if $\{\phi\} \cup \{\chi \mid \chi  \in w\}$ is inconsistent then $w$ does not have any $\phi$-successor. Let $\{\phi\} \cup \{\chi \mid \chi \in w\}$ be inconsistent. Assume that  $R^{\phi C}ww'$ for an arbitrary $w'$. By definition of $R^{\phi C}ww'$ we have that for all $\psi$, if $[\phi] \psi \in w$ then ($\psi \in w'$ and there exists $w''$ such that $R^{C}w'w''$ and $w'' = \{\phi\} \cup \{\chi \mid \neg I \chi \wedge \chi \in w\})$. By inconsistency of $\{\phi\} \cup \{\chi \mid \chi \in w\}$ we have that there exists $\chi_{1}, ..., \chi_{n}$, s.t. $\chi_{1} \wedge ... \wedge \chi_{n} \vdash \neg \phi$, where $\chi_{1}, ..., \chi_{n} \in w$. Thus, $\neg \phi \in w$. By $(nA2)$ we have $[\phi] p \wedge \neg p \in w$, and thus $p \wedge \neg p \in w'$, which contradicts the consistency of $w'$.

\end{proof}

\textbf{Proposition \ref{Prop18}.} If $R^{\phi C}ww'$, then (1) for all $w''$ if $R^{C}ww''$ then $R^{C}w'w''$, (2) there exists $w^{*} = \{\phi\} \cup \{\chi \mid \neg I \chi \wedge \chi \in w\}$ such that $R^{C}w'w^{*}$ and (3) for all $w''$ if $R^{C}ww''$ then $R^{C}w^{*}w''$.

\begin{proof}

Let $R^{\phi C}ww'$. The case (2) is straightforward from the definition of $R^{\phi C}$ with use of ($nI$). 

To prove the case (3) let us consider the worlds $w$, $w''$ and $w^{*} = \{\phi\} \cup \{\chi \mid \neg I \chi \wedge \chi \in w\}$, such that $R^{C}ww''$. Let (i) $I \alpha \in w^{*}$ and (ii) $\alpha \in w''$. Then, (iii) $\neg I \alpha \in w$, because of (ii) and $R^{C}ww''$. Having $R^{\phi C}$ means that $Cn(\phi) \cup \{\chi \mid \chi \in w\}$ is consistent. Thus, we can consistently extend $w$ with $\phi$.  From (i), by construction of $w^{*}$, we have $\vdash (\phi \wedge \chi) \rightarrow I \alpha$ for some $\chi$ s.t. $\chi \wedge \neg I \chi \in w$. This means that the extended with $\phi$ world $w$ should also contain $I \alpha$, which contradicts (iii). Thus, if $R^{C}ww''$ then $R^{C}w^{*}w''$.



For the case (1) assume that there exists $w''$ such that $R^{C}ww''$ and it is not the case that $R^{C}w'w''$, i.e., there exists $\alpha$, such that $I \alpha \in w'$ and $\alpha \in w''$. By Lindenbaum's lemma, each consistent theory can be extended to a maximal consistent theory. Thus, we extend the worlds $w$, $w'$, $w''$ and $w^{*}$ to maximal consistent theories which we denote $w_{m}$, $w'_{m}$, $w''_{m}$ and $w^{*}_{m}$, respectively. In accordance with the definitions of $R^{C}$ and $R^{\phi C}$, this amounts to say that $w'_{m} = \{\chi_{1} \mid [\phi] \chi_{1} \in w_{m}\}$, $w''_{m} = \{ \neg \chi_{2} \mid I \chi_{2} \in w_{m}\}$, and $w^{*}_{m} = \{\phi\} \cup \{\chi \mid \neg I \chi \wedge \chi \in w_{m}\}$. It is clear that $w''_{m}$ is not empty, otherwise $\alpha \not \in w''_{m}$ for any $\alpha$ and thus $R^{C}w'_{m}w''_{m}$. Having $\alpha \in w''_{m}$ means that $\alpha \in \{\neg \chi_{2} \mid I \chi_{2} \in w_{m}\}$, that is, (i) $I \neg \alpha \in w_{m}$. From $I \alpha \in w'_{m}$ and the construction of $w'_{m}$ we also have (ii) $[\phi] I \alpha \in w_{m}$. By establishing that $\{\phi\} \cup \{\chi \mid \neg I \chi \wedge \chi \in w_{m}\}$ is consistent (from the definition of the world $w^{*}_{m}$), and by recalling that we are considering maximized theories, we have (iii) $\phi \in w_{m}$. From (i), (ii), (iii) and ($AI$), we have $I (\alpha \vee \neg \alpha) \in w_{m}$, and thus $\neg (\alpha \vee \neg \alpha) \in w''_{m}$. By (\textit{A8}), we have $\neg \alpha \wedge \neg \neg \alpha \in w''_{m}$, which is a contradiction. Thus, having an assumption that $R^{C}w' w''$ does not hold, because the consistent theories $w$, $w'$, $w''$, and $w^{*}$ cannot be extended to maximal consistent theories, which contradicts the Lindenbaum's lemma. Thus, $ \alpha \not \in w''$, that is, $R^{C}w'w''$.


\end{proof}

\textbf{Proposition \ref{Prop19}.} If $R^{\phi C}ww'$ and $R^{C}w'w''$, then either $w'' \subseteq \{\phi\} \cup \{\chi \mid \neg I \chi \wedge \chi \in w\}$ or $w'' \subseteq w'''$ s.t. $R^{C}ww'''$.

\begin{proof}

Let (i) $R^{\phi C}ww'$ and (ii) $R^{C} w'w''$. Assume that (iii) $w'' \not \subseteq \{\phi\} \cup \{\chi \mid \neg I \chi \wedge \chi \in w\}$ and (iv) $w'' \not \subseteq w'''$ s.t. $R^{C}ww'''$. The cases (iii) and (iv) mean that it is not the case that $R^{C}ww''$, that is, (v) there exists $\alpha \in w''$, (vi) $I \alpha \in w$, and (vii) $\alpha \not \in \{\phi\} \cup \{\chi \mid \neg I \chi \wedge \chi \in w\}$. From (ii) and (v), we have (viii) $I \alpha \not \in w'$. From (viii) and (i), we have (ix) $\neg [\phi] I \alpha \in w$. Let all the valid formulas of $w$ constitute the set $\Gamma$. It is clear that $\Gamma \vdash I \alpha \wedge \neg [\phi] I \alpha$. By the rule ($CN$), we get $\Gamma' \vdash \phi \rightarrow \alpha$, where $\Gamma' = \{ \chi \mid \chi \wedge \neg I \chi \in \Gamma\}$. By deduction theorem we have $\Gamma' \cup \phi \vdash \alpha$, which contradicts (vii).

\end{proof}

\end{document}